\documentclass[11pt]{amsart}
\usepackage{graphicx}
\usepackage{multicol,multirow}
\usepackage{amsmath,amssymb,amsfonts}
\usepackage{rotating}
\usepackage{appendix}
\usepackage[numbers]{natbib}
\usepackage{tikz}
\usepackage{microtype, mathtools, mathrsfs, enumerate, dsfont, esvect}
\usepackage{wasysym}
\usepackage{pgf}
\usetikzlibrary{positioning,automata}
\usetikzlibrary{arrows,automata,positioning}
\usepackage{bbm}
\usepackage{verbatim}
\usepackage{url}
\newcommand{\Z}{\mathbb{Z}}

\newcommand{\calA}{\mathcal{A}}

\newcommand{\calL}{\mathcal{L}}

\newcommand{\N}{\mathbb{N}}
\newcommand{\lpq}[2]{L_{#1 \to #2}}

\newtheorem{thm}{Theorem}
\newtheorem{exam}{Example}
\newtheorem{lem}[thm]{Lemma}
\newtheorem{fact}[thm]{Fact}

		\numberwithin{thm}{section}

	\newtheorem{cor}[thm]{Corollary}

	\newtheorem*{thm*}{Theorem}
	\newtheorem*{lemma*}{Lemma}
	\newtheorem*{prop*}{Proposition}
	\newtheorem*{cor*}{Corollary}
	\newtheorem*{conj*}{Conjecture}
\theoremstyle{definition}
	
	\newtheorem*{example*}{Example}

\newtheorem{ques}[thm]{Question}

\numberwithin{thm}{section}

\theoremstyle{definition}
\newtheorem{defn}[thm]{Definition}

\newcommand{\defemph}{\textbf}

\title[A Dichotomy for $k$-automatic expansions of Presburger Arithmetic]{A Dichotomy for $k$-automatic expansions of Presburger Arithmetic \thanks{The first-named author was supported by NSERC grant RGPIN-2016-03632.  The second-named author was supported by the National Science Foundation under award No. DMS -2303368.}}

\author{Jason Bell}
\address{Department of Pure Mathematics\\ University of Waterloo\\
Waterloo, Ontario, Canada N2L 3G1}
\email{jpbell@uwaterloo.ca}

\author{Alexi Block Gorman}
\address{Universiteit van Amsterdam\\
Institute for Logic, Language and Computation\\
Science Park 900\\ 
1098 XH Amsterdam\\
Netherlands} \email{a.t.blockgorman@uva.nl}

\author{Chris Schulz}
\address{Department of Pure Mathematics\\ University of Waterloo\\
Waterloo, Ontario, Canada N2L 3G1}
\email{chris.schulz@uwaterloo.ca}


\begin{document}

\begin{abstract}
Let $k\ge 2$ and let $X$ be a subset of the natural numbers that is $k$-automatic and not eventually periodic. 
We show that the following dichotomy holds: either all $k$-automatic subsets are definable in the expansion of Presburger arithmetic in which we adjoin the predicate $X$, or $(\mathbb{N},+,X)$ has the same definable sets as $(\mathbb{N},+,k^{\mathbb{N}})$.
\end{abstract}


\maketitle

\section{Introduction}

A set $X \subseteq \mathbb{N}$ is $k$-{\bf automatic} if the language of base-$k$ expansions of elements of $X$, viewed as strings over the alphabet $\{0, 1, \ldots, k-1\}$, is accepted by some finite-state automaton. These sets, and their definability properties, have played an important role in both model theory and in theoretical computer science. A key structural insight, due to B\"uchi in \cite{Buc62} and Bruy\`ere in \cite{B85}, is that the $k$-automatic subsets of $\mathbb{N}^n$ are exactly those definable in the structure $(\mathbb{N}, +, V_k)$, where the unary function $V_k:\mathbb{N}\to \mathbb{N}$ returns the largest power of $k$ that divides $n$.

This correspondence has far-reaching consequences: it shows, for example, that $(\mathbb{N}, +, V_k)$ admits a decidable theory. The broader project of understanding expansions of Presburger arithmetic by automatic sets has a long and rich history. 
One such result is the Cobham-Semenov theorem \cite{C69} which asserts that if a set $X \subseteq \mathbb{N}^n$ is definable both in $(\mathbb{N}, +, V_k)$ and in $(\mathbb{N}, +, V_\ell)$, for multiplicatively independent bases $k$ and $\ell$, then $X$ is already definable in Presburger arithmetic; that is, it is definable in $(\mathbb{N}, +)$.
Other such results include the fact that the structure $(\N,+,V_k,V_{\ell})$ defines multiplication when $k$ and $\ell$ are multiplicatively independent, due to Villemaire in \cite{V92},
which was then strengthened by B\`es in \cite{Bes97}, in which he proves that $(\N,+,V_k,\ell^{\N})$ defines multiplication when $k$ and $\ell$ are multiplicatively independent.

There exists a natural hierarchy of structures expanding Presburger arithmetic by non-Presburger-definable $k$-automatic sets. At the top lies $(\mathbb{N}, +, V_k)$, which completely captures $k$-automaticity and exhibits some degree of model-theoretic complexity (e.g., it has TP$_2$). At the bottom lies $(\mathbb{N}, +, k^{\mathbb{N}})$, where $k^{\mathbb{N}} = \{ k^n : n \in \mathbb{N} \}$, a reduct that has NIP and in which the definable sets are comparatively tame (see Semenov \cite{Sem80}, and Lambotte--Point \cite[Theorem~2.32]{LP20}).

In this paper, we show that for expansions of $(\mathbb{N}, +)$ by a single unary $k$-automatic set $X \subseteq \mathbb{N}$ that is not eventually periodic, the hierarchy collapses. Specifically, we prove the following dichotomy holds.
\begin{thm} \label{thm:main}
Let $X\subseteq \mathbb{N}$ be a $k$-automatic set that is not eventually periodic. Then either $(\mathbb{N},+,X)$ defines all $k$-automatic sets, or $X$ is definable in $(\mathbb{N},+,k^{\mathbb{N}})$. 
\end{thm}

The main ingredients for proving this theorem are characterizations of sparse and non-sparse languages, as well as the definable sets of $(\N,+,k^{\N})$ and $(\N,+)$ respectively.
In particular, we use that sets definable from sparse ones are precisely those definable in $(\N,+,k^{\N})$. 
On the other hand, sets not definable from a sparse one necessarily define an additive basis\footnote{A set $X$ such that for some $n \in \N$, every number can be written as a sum of at most $n$ elements of $X$.} for the natural numbers that is not eventually periodic and from such a set we show one can define all automatic sets.
We also make key use of the fact that to check whether a $k$-automatic set $X$ is Presburger-definable (i.e., definable in $(\N,+)$) it suffices to check only that it has period $p\geq 1$ on an interval that is bounded in terms of $p$ and $X$.

The outline of this paper is as follows.  In \S \ref{sec:pre} we provide background on automata and Presburger definability, in addition to other relevant concepts required to prove Theorem \ref{thm:main}.  In \S \ref{sec:cyc} we prove Theorem \ref{thm:main} in a key special case: namely, when $X$ is the set of natural numbers whose base-$k$ expansions are accepted by a cycle language; that is, a regular language accepted by a deterministic finite-state automaton with one final state that is equal to its initial state (see \S \ref{sec:pre} for definitions). In \S \ref{sec:gen} we prove the general case. Finally in \S \ref{sec:fur}, we make some observations that follow from our main result and a direction for further inquiry.

\section{Preliminaries}
\label{sec:pre}

\subsection{Presburger arithmetic}

Presburger arithmetic refers to the first-order theory of natural numbers with addition. It was introduced by Presburger in \cite{P29}, who provided a recursive axiomatization and proved that the theory is both consistent and complete.

\begin{defn}
    \defemph{Presburger arithmetic} is the first-order theory of the structure $(\mathbb{N}, +)$.
\end{defn}

Throughout this paper, we consider first-order formulas in Presburger arithmetic only in terms of the sets they define in the standard model. Non-standard models will not concern us.

There is a number-theoretic characterization of the sets definable in Presburger arithmetic, due to Ginsburg and Spanier \cite{GS66S}.

\begin{defn}
    A subset $X$ of $\mathbb{N}^d$ is said to be \defemph{linear} if it is of the form:
    \[ X = \left\{x + \sum_{i=1}^n k_i y_i : k_1, \dots, k_n \in \mathbb{N} \right\} \]
    where $x, y_1, \dots, y_n$ are fixed vectors in $\mathbb{N}^d$.

    A subset of $\mathbb{N}^d$ is \defemph{semilinear} if it is a finite union of linear sets.
\end{defn}

\begin{thm}[\cite{GS66S}]
    A subset of $\mathbb{N}^d$ is definable in Presburger arithmetic if and only if it is semilinear.
\end{thm}

In this paper, we focus on the one-dimensional case. In this setting, semilinear sets correspond precisely to eventually periodic sets:

\begin{defn}
    A subset $X$ of $\mathbb{N}$ is \defemph{eventually periodic} if there exist $N, p \in \mathbb{N}$ such that $n \in X$ if and only if $n + p \in X$ for all $n \geq N$.
\end{defn}
Thus we have the following result, which also follows from quantifier elimination for Presburger arithmetic (see Theorem 32F in \cite{E01}).
\begin{cor}
    Let $X \subseteq \mathbb{N}$. Then $X$ is Presburger-definable if and only if $X$ is eventually periodic.
\end{cor}

\subsection{Finite automata}

This paper will study regular languages by way of the corresponding finite automata. As such, we begin with a definition of the non-deterministic finite automaton:

\begin{defn}
    A \defemph{non-deterministic finite automaton} (or \defemph{NFA}) is a tuple $\calA = (Q, I, \Sigma, \delta, F)$ where:

    \begin{itemize}
        \item $Q$ is a finite set of elements called \defemph{states};
        \item $I \subseteq Q$ is the \defemph{set of start states} or \defemph{initial states};
        \item $\Sigma$ is a finite \defemph{alphabet};
        \item $\delta : Q \times \Sigma \to \mathcal{P}(Q)$ is a function called the \defemph{transition function};
        \item $F \subseteq Q$ is a set of \defemph{accept states} or \defemph{final states}.
    \end{itemize}
\end{defn}
    
\begin{defn}
    An NFA $\calA = (Q, I, \Sigma, \delta, F)$ is called \defemph{deterministic}, or a \defemph{DFA}, if $I$ has exactly one element and if $\delta(q, y)$ has at most one element for all $q \in Q, y \in \Sigma$.
\end{defn}

We first recall that for an NFA $\calA = (Q, I, \Sigma, \delta, F)$, we can inductively extend the function $\delta : Q \times \Sigma \to \mathcal{P}(Q)$ to a function from $Q \times \Sigma^*$ to $\mathcal{P}(Q)$ by declaring that, for a word $v\in \Sigma^*$ and $y\in \Sigma$:

$$\delta(q, vy) = \bigcup_{q' \in \delta(q, v)} \delta(q',y)$$

If $\calA$ is deterministic, this extended $\delta$ maintains the same property that its output is always at most a singleton set.

A \defemph{run} of $\calA$ on a word $w \in \Sigma^*$ is a finite sequence $(q_0, q_1, \dots, q_{|w|})$ of states, with $q_0 \in I$ a start state, such that for $1 \leq i \leq |w|$ we have $q_i \in \delta(q_{i-1}, w_i)$. A run is \defemph{accepting} if $q_{|w|} \in F$; if any accepting run exists, we say that $\calA$ \defemph{accepts} $w$. If $\calA$ does not accept a word $w$, we say it \defemph{rejects} $w$.  We may also speak of runs \defemph{from} states other than initial states; a run from $q$ means a run in the automaton $\calA'$ that replaces $I$ with $\{q\}$.

A language $L$ is regular if and only if there exists a finite automaton $\calA$ such that $L$ is the set of words accepted by $\calA$.
It is a well-known fact due to Rabin and Scott that we may without loss of generality take $\calA$ to be deterministic:

\begin{fact}
    Let $\calA$ be any NFA. There exists a deterministic finite automaton $\calA'$ such that $\calA$ and $\calA'$ accept precisely the same words.
\end{fact}

We will also need two facts about the sizes of deterministic finite automata after transformations are applied to the corresponding languages:

\begin{lem}
\label{numstates_prefix}
    Let $L$ be a regular language whose deterministic finite automaton $\calA$ has $M$ states. Let $u$ be any word. Then the language $\{v : uv \in L\}$ has a deterministic finite automaton with $(M+1)$ states.
\end{lem}
\begin{proof}
    Let $\calA = (Q, I, \Sigma, \delta, F)$. We define the DFA $\calA'$ on $\Sigma$ as follows. 
    First, add a new state $q_u$ to $Q$ to obtain the set of states in $\calA'$; $q_u$ will be the initial state. 
    The final states of $\calA'$ will be those in $F$; if (and only if) $u \in L$, we also add $q_u$ as a final state. 
    Finally, the transition function of $\calA'$ agrees with that of $\calA$ on its domain, and for each symbol $\sigma \in \Sigma$, we add a transition from $q_u$ to $\delta(I, u\sigma)$ on $\sigma$.

    We claim that $\calA'$ recognizes $\{v : uv \in L\}$. Let $uv \in L$. If $v$ is empty, then $u \in L$ and hence the empty word is recognized by $\calA'$. Otherwise, let $v = \sigma w$. Then $\delta(I, u\sigma w) \in F$; so $\delta(\delta(I, u\sigma), w) \in F$; so $\delta(q_u, \sigma w) \in F$, as needed. The proof when $uv \notin L$ is analogous.
\end{proof}

\begin{lem}
\label{numstates_iff}
    Let $L_1$ be a regular language whose deterministic finite automaton $\calA_1$ has $M_1$ states, and let $L_2$ be a regular language whose deterministic finite automaton $\calA_2$ has $M_2$ states. Then the language $\{w : w \in L_1 \leftrightarrow w \in L_2\}$ has a deterministic finite automaton with $M_1 \cdot M_2$ states.
\end{lem}
\begin{proof}
    Let $\calA_1 = (Q_1, I_1, \Sigma, \delta_1, F_1)$ and $\calA_2 = (Q_2, I_2, \Sigma, \delta_2, F_2)$. We will construct a new automaton $\calA_3 = (Q_3, I_3, \Sigma, \delta_3, F_3)$. Let $Q_3 = Q_1 \times Q_2$, $I_3 = I_1 \times I_2$, and let $\delta_3((q_1, q_2), \sigma) = \delta_1(q_1, \sigma) \times \delta_2(q_2, \sigma)$. Note that this definition naturally extends to words; $\delta_3((q_1, q_2), w) = \delta_1(q_1, w) \times \delta_2(q_2, w)$.

    Let $F_3 = \{(q_1, q_2) \in Q_3 : q_1 \in F_1 \leftrightarrow q_2 \in F_2\}$. Then we note that $\calA_3$ recognizes $w$ iff the single element of $\delta_3((q_1, q_2), w)$ is in $F_3$. This is true if and only if the equivalence $\delta_1(q_1, w) \in F_1 \leftrightarrow \delta_2(q_2, w) \in F_2$ holds; but this is precisely when $w \in L_1 \leftrightarrow w \in L_2$.
\end{proof}

Call an NFA \defemph{trim} if every state can be reached by some path from the start state and for each state there exists a final state that can be reached from the state in question.
Call an NFA or DFA recognizing the language $L$ \defemph{minimized} if it has the smallest number of states possible among all automata that recognize $L$ (within the class of NFAs and DFAs respectively).
Note that for any NFA, there exists an equivalent trim automaton and an equivalent minimized automaton, and these can be made deterministic up to allowing the transition function to be a partial function.

\begin{defn}\label{def:digraph}
    The \defemph{digraph structure} of a finite automaton $\calA = (Q, I, \Sigma, \delta, F)$ is the directed graph $(Q, \Delta)$ where $p \Delta q \iff \exists \sigma \in \Sigma : q \in \delta(p, \sigma)$.
    
    We may use digraph terminology to apply properties of the digraph structure to the automaton. For example, a \defemph{strongly connected component} of $\calA$ is a maximal subset $S \subseteq Q$ such that for all $p, q \in S$ there exists a path from $p$ to $q$, or equivalently a word $w$ such that $\calA$ runs from $p$ to $q$ on $w$. We define a \defemph{leaf} of the digraph structure to be a strongly connected component with no outgoing transitions to states in a different strongly connected component.
    
\end{defn}

Within this paper, it will be very useful to examine the internal structure of an automaton in terms of which words run to which individual states. So we will define two classes of languages:

\begin{defn}
    Let $\calA = (Q, I, \Sigma, \delta, F)$ be a finite automaton and $p, q \in Q$. The \defemph{path language} $\lpq{p}{q} \subseteq \Sigma^*$ is the set of words $w$ on which there is a run from $p$ to $q$ in $\calA$. In the case that $p = q$, we call this the \defemph{cycle language}. 
\end{defn}

It is elementary that path languages are always regular. Moreover, note that $\lpq{q}{q} = \lpq{q}{q}^*$.

\begin{defn}
    An \textbf{induced subautomaton} of an automaton $\calA=(Q,I,\Sigma,\delta,F)$ corresponding to some $Q' \subseteq Q$ is the automaton $(Q', I', \Sigma, \delta|_{Q'}, F \cap Q')$ where $I'$ is the union of $Q' \cap I$ along with all elements of $Q'$ to which there is an incoming transition from outside of $Q'$.
\end{defn}

\subsection{Encoding sets of natural numbers}

Our results in this paper relate to the encoding of natural numbers via the base-$k$ expansion. We define this, and our notations related to it, here.

\begin{defn}
Given a natural number $k\ge 2$, we let $\Sigma_k=\{0,\ldots ,k-1\}$.
    A \textbf{base-$k$ expansion} of a natural number $x$ is a word $w_n w_{n-1} \dots w_0$ over the alphabet $\Sigma_k$ such that:

    $$x = \sum_{i=0}^n w_i k^i$$

    In this case, we say that $x = [w]_k$.
\end{defn}

Note that the above definition defines $[w]_k$ by treating $w$ as a \textbf{most-significant-digit-first} (MSD-first) representation. Some previous work has used the opposite \textbf{least-significant-digit-first} representation, so in order to ensure that the reader understands what our convention is in this paper, we instruct the reader to assume MSD-first representations are in use in all cases. By example: $[120]_3 = 1 \cdot 3^2 + 2 \cdot 3^1 = 15$, not $7$.

We will commonly apply this notation to sets as well. So given a language $L\subseteq \Sigma_k^*$, we let $[L]_k$ denote the set of natural numbers of the form $[w]_k$ with $w\in L$. 
We say that a set $X\subseteq \mathbb{N}$ is $k$-{\bf automatic} if there is a regular sublanguage $L$ of $\Sigma_k^*$ such that $X=[L]_k$.

We note that there is a well-known dichotomy for regular languages:
if $L$ is a regular language and if 
$|L_{\le n}|$ denotes the number of words in $L$ 
of length at most $n$, then either $|L_{\le n}|$ is  polynomially bounded in $n$, or else there is a 
constant $C>1$ such that $|L_{\le n}|\ge C^n$ 
for infinitely many $n$.
More generally, this dichotomy can be stated as follows:
\begin{thm}[Proposition 7.1, \cite{BM19}]\label{sparse_char}
    Let $L$ be a regular language. The following are equivalent:
    \begin{enumerate}
\item $f_{\calL}(n) = O(n^d)$ for some natural number d.
\item  $f_{\calL}(n) = o(C^n)$ for every C > 1.
\item  There do not exist words $u, v, a, b$ with $a, b$ non-trivial and of the same length
and $a \neq b$ such that $u\{a,b\}^{*} v \subseteq L$.
\item  Suppose $\Gamma = (Q, \Sigma, \delta, q_0, F )$ is an automaton accepting L in which all states are accessible. Then $\Gamma$ satisfies the following.\\
$(*)$ If $q$ is a state such that $\delta(q,v) \in F$ for some word $v$ then there is at
most one non-trivial word $w$ with the property that $\delta(q,w) = q$ and
$\delta (q, w') \neq q$ for every non-trivial proper prefix $w'$ of $w$;
\item  There exists an automaton accepting $L$ that satisfies $(*)$.
\item  $L$ is a finite union of languages of the form $v_1w_1^* v_2w_2^* \cdots v_mw_m^* v_{m+1}$ where
$k \geq 0$ and the $v_i$'s are possibly trivial words and the $w_i$'s are non-trivial words.
\end{enumerate}
\end{thm}

We say that a regular language $L$ is 
{\bf sparse} if $|L_{\le n}| \in O(n^d)$ for some 
natural number $d$, and we will say that a $k$-
automatic subset $X$ of $\N$ is {\bf sparse} if 
$X=[L]_k$ for some sparse regular sublanguage of 
$\Sigma_k^*$.

There is a fundamental connection between $k$-automatic sets and the definable sets in an expansion of Presburger arithmetic, which was discovered by B\"uchi \cite{Buc60} and formally proven by Bruy\`ere \cite{B85}:

\begin{fact}
    Let $V_k$ send every positive integer $x$ to the largest power of $k$ dividing $x$. A set $X \subseteq \N$ is $k$-automatic if and only if it is definable in the structure $(\N, +, V_k)$.
\end{fact}

Therefore, the definable sets in $(\N, +, V_k)$ can be understood via the theory of finite automata. This result is fundamental to the study of base-$k$ 
arithmetic; for instance, it gives us an effective correspondence between formulas in the language of $(\N, +, V_k)$ and finite automata, and hence a 
decision algorithm in this language. The reducts of this structure are the subject of this paper.

We shall need one more result, proven by B\`es in \cite{Bes97}:

\begin{fact}[\cite{Bes97}, Theorem 3.1]
    \label{bes_define_kn}
    Let $X \subseteq \N$ be $k$-automatic but not definable in $(\N, +)$ (i.e. not eventually periodic). Then the structure $(\N, +, X)$ defines the set $k^\N = \{k^n : n \in \N\}$.
\end{fact}

\subsection{Semenov's characterization}
\label{ss:sem}

Semenov \cite{Sem80} gives a characterization of the one-dimensional sets definable from $(\N, +, k^\N)$, which will be very useful to us. We first define the following subclass of regular languages:

\begin{defn}
    Let $k$ be a natural number. A $k$-\textbf{Semenov-type} language is a sublanguage of $\{0,\ldots ,k-1\}^*$ of the form:

    $$u_1 v_1^* u_2 v_2^* \dots u_p v_p^* \Sigma_{\ell,m,c}$$

    where $p, \ell, m, c \in \N$, $u_1, \dots, u_p, v_1, \dots, v_p$ are finite words over the alphabet $[k]:=\{0,\ldots ,k-1\}$, and $\Sigma_{\ell,m,c}$ is the set of all words of length a multiple of $\ell$ that are the base-$k$ expansion (\textit{without} leading zeros) of a word congruent to $-c$ modulo $m$, where we take $\Sigma_{\ell,m,c}$ to be the empty language when $m=0$ or $\ell=0$.
    We suppose without loss of generality that $\Sigma_{\ell,m,c}$ does not include leading zeros, increasing $p$ if necessary so that any string of zeros preceding an element of $\Sigma_{\ell,m,c}$ is absorbed into $v_p^*$.
\end{defn}

Theorem \ref{sparse_char} allows us to observe the following fact.

\begin{fact}[\cite{Sem80}]\label{semenov}
    A one-dimensional set $X \subseteq \N$ is definable in $(\N, +, k^\N)$ if and only if $X = [L]_k$, where $L$ is a finite union of $k$-Semenov-type languages.  
    In particular, if $U$ and $V$ are respectively a sparse regular sublanguage of $\{0,\ldots ,k-1\}^*$ and a sublanguage of $\{0,\ldots ,k-1\}^*$ consisting of base-$k$ expansions (without leading zeros) of a set definable in $(\N,+,k^\N)$, then $[UV]_k$ is definable in $(\N, +, k^\N)$.
\end{fact}

\section{Proof of the Dichotomy for cycle languages}\label{sec:cyc}

In this section, we'll prove the following result.
    \begin{thm}\label{thm:reduction1}
        Let $\mathcal{A}=(Q,\{q_0\},\Sigma_k, \delta,W)$ be a DFA which reads strings left-to-right. If $p\in Q$ and 
        $X=\{[w]_k \colon w\in L_{p\to p}\}$, then either $X$ is definable in  $(\mathbb{N},+,k^{\N})$ or $V_k$ is definable in $(\N,+,X)$.
    \end{thm}

We give an overview of how we shall prove Theorem \ref{thm:reduction1}.  
We show that if $X$ is a set of natural numbers that is not Presburger definable and that corresponds to a cycle language for an automaton with input alphabet $\Sigma_k$ for some $k$, then (after replacing $k$ by a power if necessary) we may assume 
that there is a distinguished letter $a\in \{0,\ldots ,k-1\}$ such that we can define a function $F: X\to k^{\mathbb{N}}$, which is a very close definable approximation to the function $V_{k,a}|_X$, where $V_{k,a}$ is the map that takes a natural number $n$ as input and outputs $k^i$, where $i$ is the unique nonnegative integer with the property that the base-$k$ expansion of $n$ ends with exactly $i$ copies of the digit $a$.

Having defined this map $F$, we show that we can extend $F$ to a definable approximation of the function $V_{k,a}$ on $X_{\le P}$, where $P$ is a natural number and $X_{\le P}$ is the set of all numbers that are sums of at most $P$ elements of $X$.  Now we use a theorem of the first-author, Hare, and Shallit \cite{BHS18} to show that if $X$ is not sparse, then we can define an approximation of $V_{k,a}$ on all of $\N$.  From here, we can show one can define $V_k$ and so we deduce that one can define all $k$-automatic sets from $X$.
    
    To prove Theorem \ref{thm:reduction1} we may assume that $X$ is the set of base-$k$ expansions of a non-empty cycle language, and hence there is some word $w$ such that $\delta(p,w)=\{p\}$ for a fixed state $p \in Q$.  We henceforth let 
    \begin{equation}
L := L_{p\to p} =\{w\colon \delta(p,w)=\{p\}\},
\end{equation}
    and so $X=[L]_k$.

   Since $L$ is non-empty, there is some non-empty word $w\in\Sigma_k^*$ such that $\delta(p,w)=\{p\}$. 
    Possibly by replacing $k$ by a power, we may assume without loss of generality that $w$ is a letter $a\in \{0,1,\ldots ,k-1\}$. 
    Furthermore, the pigeonhole principle gives that there is some positive integer $i$ such that $\delta(q,0^i)=\delta(q,0^{2i})$ and $\delta(q, (k-1)^i)=\delta(q,(k-1)^{2i})$ for all $q\in Q$.
    Thus we may replace $k$ by a power and henceforth assume the following hold:
    \begin{enumerate}
        \item[(i)] There is a distinguished letter
        $a\in \{0,\ldots ,k-1\}$ with the property that
    \begin{equation} \delta(p,a)=\{p\}.
    \end{equation}
\item[(ii)] We have $\delta(q,00)=\delta(q,0)$ for all $q\in Q$.
\item[(iii)] We have $\delta(q,(k-1)(k-1))=\delta(q,k-1)$ for all $q\in Q$.
\item[(iv)] We also assume that $\calA$ is minimal.
    \end{enumerate}
We shall say that a word $w \in \Sigma^*$ is {\bf idempotent} (with respect to the automaton $\mathcal{A}$) if for all states $q \in Q$, we have $\delta(q, ww) = \delta(q, w)$. That is, applying the transition corresponding to $w$ twice has the same effect as applying it once. In particular, in our setting, we assume that the words $0$ and $k-1$ are idempotent in this sense.

    We now let $f$ denote the characteristic function of the set $X$.  
    Then $f:\mathbb{N}\to \{0,1\}$ is $k$-automatic and hence there are finitely many maps $f=f_1,\ldots , f_s$ such that for each $c\ge 0$ and each $j\in \{0,\ldots, k^c-1\}$ we have
    $f(k^cn+j)=f_i(n)$ for some $i\in \{1,\ldots ,s\}$ which depends on $c$ and $j$, by \cite[Theorem 6.6.2]{AS03}.
    We call these maps $f_1,\ldots ,f_s$, the {\bf kernel} of the map $f$.  
    We may assume without loss of generality that $X$ is neither sparse nor eventually periodic.  
    In particular, since $X$ is not Presburger definable, $(\mathbb{N},+,X)$ defines the set $k^{\N}$ by Fact \ref{bes_define_kn}.  
    In addition, there exists some natural number $b$ such that each $f_i$ can be realized as $f(k^r n+j)$ for some $r,j$ with $r\le b$. 

    For the remainder of the proof, we divide the proof into three cases:
    \begin{enumerate}
        \item[(I):] $\delta(p,0)=\{p\}$;
        \item[(II:)] $\delta(p,0)=\varnothing$ or $\delta(p,0)=\{p'\}$, where $p'$ is in a different strongly connected component from $p$;
        \item[(III):] $\delta(p,0)=\{p'\}$ where $p'$ and $p$ are in the same strongly connected component.
    \end{enumerate}
  We note that in Case (I) we can take $a=0$ and this simplifies the proof of this case somewhat, but we will handle cases (I) and (II) simultaneously.
    \subsection{Proof in cases (I) and (II)}
Throughout this subsection we assume the following:
\begin{itemize}
    \item One of cases (I) or (II) above holds;
    \item $X$ is not definable in $(\N,+,k^{\N})$.
\end{itemize} 
It follows that $k^{\N}$ is definable in $(\N,+, X)$ by Fact \ref{bes_define_kn}.

We now define a series of functions $F_R: X\to k^{\N}$ for $R \in \N$ as follows. If $0\in X$, we take $F_R(0)=1$, and for $n\ge 1$ in $X$ we take $F_R(n)$ to be the largest $k^i\in k^{\N}$ such that $k^{i-1}\le n$ and for all $0 \leq r \leq R$ and for all words $v$ with $|v|=j\le r+i$, we have:

\begin{align}
k^r n - [a^{j-r} 0^r]_k + [v]_k &\in X \iff v\in L \qquad \text{if } j > r, \label{eq:gr} \\
k^r n + [a^{r-j} v]_k &\in X \iff v\in L \qquad \text{if } j \le r. \label{eq:lr}
\end{align}

We note two things about $F_R$:

\begin{itemize}
    \item For $n\ge 1$, $k^0$ (i.e., $i=0$) has the property that for all $r\ge 0$ and all words of length at most $r$ with $|v|\le r$, the condition in Equation (\ref{eq:lr}) holds, and so $F_R(n)\ge 1$ for all $n$. Therefore, $\min_R F_R(n)$ is well-defined; we denote this $F(n)$.
    \item Note that because $F_R$ is defined as the largest $k^i$ such that a universal quantification over $r \leq R$ holds, we must have that $(F_R(n))_R$ is monotone decreasing (in $R$).
\end{itemize}

We will now show that, due to the pumping lemma, the point at which the sequence $(F_R(n))_R$ stabilizes is uniform in $n$.

\begin{lem}
    Let $\calA'$ be a deterministic finite automaton recognizing $L$ with set of states $Q$, and let $M = (|Q|+1)^2$. Then $F_M$ and $F$ are identical.
\end{lem}
\begin{proof}
    Fix $n$. By the above remarks, this is equivalent to the claim that for all $R > M$, $F_{M}(n) = F_{R}(n)$.

    We proceed by induction on $R$. Assume that $F_{M}(n) = F_{M+1}(n) = \dots = F_{R-1}(n)$. We need to show that $F_{R-1}(n) = F_{R}(n)$. Let $F_{R-1}(n) = k^i$; it suffices to show that for $r = R$ and for all words $v$ with $|v| = j \leq r + i$, the appropriate condition (either (3) or (4)) is satisfied.

    We first consider the case where $|v| \leq r$. In this case, we need to show that $k^r n + [a^{r-|v|}v]_k \in X \iff v \in L$. Let $[u]_k = n$; then equivalently, we want $ua^{r-|v|}v \in L \iff v \in L$.

    Let $L' = \{a^{r'-|v'|}v' : ua^{r'-|v'|}v' \in L\}$. By Lemma \ref{numstates_prefix}, the automaton for $L'$ has at most $M$ states.
    
    Consider the string $a^{r-|v|}v$. 
    Note that this string has length $r = R > M$, so by the pumping lemma, we may write $a^{r-|v|}v = xyz$ where $y$ is nonempty and each $xy^mz \in L' \iff xyz \in L'$. Then $xz = a^{r'-|v'|}v'$ for some $r' < r$. 
    We know that $ua^{r'-|v'|}v' \in L \iff v \in L$ by inductive assumption, so $ua^{r-|v|}v \in L \iff xyz \in L' \iff xz \in L' \iff ua^{r'-|v'|}v' \in L \iff v \in L$.

    Now we consider the case where $|v| > r$. In this case, we need to show that $k^r n - [a^{j-r} 0^r]_k + [v]_k \in X \iff v \in L$. Let $v = v_0 v_1$ where $|v_1| = r$; then the condition $k^r n - [a^{j-r} 0^r]_k + [v]_k \in X$ may be rewritten as $k^r (n - [a^{j-r}]_k + [v_0]_k) + [v_1]_k$. Let $w_{j,v_0}$ be a word such that $[w_{j,v_0}]_k = n - [a^{j-r}]_k + [v_0]_k$; then we want to show that $w_{j,v_0} v_1 \in L \iff v_0 v_1 \in L$.

    In this case, we set $L'_{j,v_0} = \{v_1 : w_{j,v_0}v_1 \in L \leftrightarrow v_0 v_1 \in L\}$. By Lemmas \ref{numstates_prefix} and \ref{numstates_iff}, the automaton for $L'_{j,v_0}$ has at most $M$ states. Because $|v_1| = r = R > m$, we apply the pumping lemma and let $v_1 = xyz$ with $y$ nonempty and $xy^mz \in L'_{j,v_0} \iff xyz \in L'_{j,v_0}$ for all $m$, in particular $xz \in L'_{j,v_0} \iff xyz \in L'_{j,v_0}$.
    
    Let $v_1' = xz$, and let $|v_1'| = r' < r$. Let $j' = j + r' - r$. Then:

    \begin{align*} 
    [w_{j,v_0}v_1']_k &= k^{r'}(n - [a^{j-r}]_k + [v_0]_k) + [v_1']_k \\
    &= k^{r'}n - k^{r'}[a^{j-r}]_k + k^{r'}[v_0]_k + [v_1']_k \\
    &= k^{r'}n - k^{r'}[a^{j-r}]_k + [v_0v_1']_k \\
    &= k^{r'}n - k^{r'}[a^{j'-r'}]_k + [v_0v_1']_k \\
    &= k^{r'}n - [a^{j'-r'}0^{r'}]_k + [v_0v_1']_k \\
    \end{align*}

    By inductive assumption, $k^{r'}n - [a^{j'-r'}0^{r'}]_k + [v_0v_1']_k \in X \iff v_0 v_1' \in L$, i.e. $w_{j,v_0}v_1' \in L \iff v_0 v_1' \in L$. Therefore, $xz = v_1' \in L'_{j,v_0}$. It follows that $xyz = v_1 \in L'_{j,v_0}$, so $w_{j,v_0} v_1 \in L \iff v_0 v_1 \in L$, as required.
\end{proof}

  Intuitively, $F$ is trying to capture (imperfectly) the length of the largest suffix of the base-$k$ expansion that consists entirely of $a$'s.  We note that using geometric series, it is not difficult to show that the limiting map $F$ is definable in $(\N,+,X)$:
  
  \begin{lem}
  Under the assumptions above, the map $F$ is definable in $(\N,+,X)$.
  \end{lem}
  \begin{proof}
  We will construct, step-by-step, several definable functions and relations that ultimately give us $F_R$ for a fixed $R$ when put together. As $F = F_M$ by the previous lemma, we conclude that $F$ itself is definable.
  Throughout this proof, we work in the expanded language $\{+,-,\equiv_n, (x \mapsto nx)_{n \in \N},<,X, k^\N\}$.
    We do this without loss of generality because it is elementary to show that the relations $<, >, \leq, \geq, \neq$, modular equivalence $\equiv_n$ modulo a constant $n$, multiplication $nx$ by a constant $n$, the graph of the partial function $-$, and all natural number constants themselves are definable in Presburger arithmetic, and because Fact \ref{bes_define_kn} guarantees that we can define $k^\N$ in $(\N,+,X)$.

    \begin{itemize}
        \item Let $\ell(x)$ be the smallest $k^i \in k^\N$ such that $x < k^i$. The graph of this function is defined by the formula $y \in k^\N \wedge \forall p_i \in k^\N \: (p_i > x \leftrightarrow p_i \geq y)$. Note that if $v$ starts with no leading zero, $\ell([v]_k)$ is encoded by the word $10^{|v|}$.
        \item Let $L(x, y)$ be the binary relation that holds precisely when there is a word $v$ with length $j$ such that $[v]_k = x$, $k^j = y$, and $v \in L$. The definition of this relation depends on whether we are in case (I) or (II). If case (I) holds, then for all $v$, we have $v \in L \iff 0v \in L$. Therefore, in this case, $L(x, y)$ is defined by the formula $y \geq \ell(x) \wedge x \in X$. If case (II) holds, then no word beginning with $0$ is in $L$. Therefore, in this case, $L(x, y)$ is defined by the formula $y = \ell(x) \wedge x \in X$.
        \item Let $a(x, y)$ be the partial function where, if $x = k^i$ and $y = k^j$, then $a(x, y) = [a^{i-j}0^j]_k$. For example, $a(7, 4) = [aaa0000]_k$. Notice that $a(i, j) = [a]_k (k^j + \dots + k^{i-1}) = [a]_k \frac{k^i-k^j}{k-1}$. Because $[a]_k$ and $(k-1)$ are constants, we therefore define the graph of $a$ by the formula $x \in k^\N \wedge y \in k^\N \wedge (k-1) z = [a]_k (x - y)$.
        \item Fix an arbitrary constant $r$. Let the relation $E_{(3),r}(x, y, n)$ hold if there is a word $v$ with length $j$ such that $[v]_k = x$, $k^j = y$, and equivalence (3) holds. Note that $k^r$ is a constant here. Then $E_{(3),r}$ is defined by the formula $y \geq \ell(x) \wedge (k^r n + x - a(y, k^r) \in X \leftrightarrow L(x, y))$.
        \item Fix an arbitrary constant $r$. Let the relation $E_{(4),r}(x, y, n)$ hold if there is a word $v$ with length $j$ such that $[v]_k = x$, $k^j = y$, and equivalence (4) holds. Note that $k^r$ is a constant here. Then $E_{(4),r}$ is defined by the formula $y \geq \ell(x) \wedge (k^r n + a(k^r, y) + x \in X \leftrightarrow L(x, y))$.
        \item Fix an arbitrary constant $r$. Let the relation $E_r(x, y, n)$ hold if there is a word $v$ with length $j$ such that $[v]_k = x$, $k^j = y$, and equivalence (3) or (4) holds depending on whether $j > r$. Then $E_r$ is defined by the formula $(y > k^r \wedge E_{(3),r}(x, y, n)) \vee (y \leq k^r \wedge E_{(3),r}(x, y, n))$.
        \item Let the relation $A_r(n, z)$ hold if $z = k^i$, $k^{i-1} \leq n$, and for all $v$ with $|v| = j \leq r + i$, we have $E_r(x, y, n)$ when $[v]_k = x$ and $k^j = y$. Then $A_r$ is defined by the formula $z \in k^\N \wedge \ell(n) \geq z \wedge \forall x \: \forall y \: ((y \in k^\N \wedge \ell(x) \leq y \wedge y \leq k^r z) \to E_r(x, y, n))$.
        \item Let the relation $A_{\leq R}(n, z)$ hold if $A_{r}(n, z)$ holds for all $r \leq R$. Then for a fixed $R$, $A_{\leq R}$ is defined by the formula $A_0(n, z) \wedge A_1(n, z) \wedge \dots \wedge A_R(n, z)$.
        \item The function $F_{R}(n)$ returns the largest $z = k^i \in k^\N$ such that $k^{i-1} < n$ (equivalently $k^i < kn$) and such that $A_{\leq R}(n, k^i)$ holds. Then the graph of $F_{R}$ is defined by the formula $z \in k^\N \wedge z < kn \wedge A_{\leq R}(n, z) \wedge \forall z' \: ((z' \in k^\N \wedge z' < kn \wedge A_{\leq R}(n, z')) \to z' \leq z)$.
    \end{itemize}
  \end{proof}
 \begin{exam}
 We give an example to show how to compute the function $F$ in a special case.  Consider the case when $k=3$ and when our distinguished letter is $a=1$ and when $X$ is the set of natural numbers whose base-$k$ expansion is in the cycle language of the state $q_2$ of the automaton in Figure \ref{fig:Fexam}, where we again read strings from left-to-right.  We claim that $F(22)=F([211]_3)=3$.
 \end{exam}
 
\begin{figure}[!htbp] 
\begin{center}
\begin{tikzpicture}[shorten >=1pt,node distance=2cm,on grid]
  \node[state]   (q_0)                {$q_0$};
  \node[state]           (q_1) [right=of q_0] {$q_1$};
    \node[state, accepting]     (q_2) [right=of q_1] {$q_2$};

  \path[->] (q_0)		(q_0) edge [loop below]   node         {0,2} ()
   (q_2)  (q_2) edge [loop above] node {0,1} () 
   (q_1)  (q_1) edge [loop above] node {0,2} () 
  (q_0)  (q_0) edge [bend right=50] node [below] {1} (q_2) 
  (q_1) (q_1) edge [bend left=50] node [above] {1} (q_2)
  (q_2) (q_2) edge node [below] {2} (q_1);
  \end{tikzpicture}
  \end{center}
  \caption{An automaton with cycle language at $q_2$ having $F(22)=3$.}
  \label{fig:Fexam}
\end{figure}
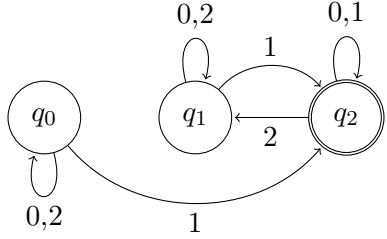

Then if we take $n=[211]_3=22$, $F(n)=3$.  
To see this, notice that if $r\ge 0$ and $v$ is a word with length at most $r+1$ with $|v|= r+1$, then 
$3^r n - [1^{|v|-r}0^r]_3 + [v]_3$ has ternary expansion
$21v$, which is in the cycle language of the state $q_2$ if and only if $v$ is.  
If $|v|\le r$, then 
$$3^r n + [1^{r-j} v]_k$$ has ternary expansion $21^{2+r-|v|} v$, which is again in the cycle language if and only if $v$ is.  It follows that $F(22)\ge 3$.  To see that $F(22)<9$, observe that for $r=0$, $v=00$, then $n - [11]_3+[v]_3 = [200]_3$, which is not in $X$, while $[00]_3$ is.  Thus $F(22)=3$, as claimed.

We now return to the proof of Theorem \ref{thm:reduction1} in cases (I) and (II). By definition of the map $F$, we also have implication:
\begin{equation} \label{eq:i+1}
n\le k^i \implies F(n)\le k^{i+1},
\end{equation} for $n\in X$.  

As mentioned earlier, we view the function $F$ as a (definable) proxy for the function $V_{k,a}$ on the set $X$, where
$V_{k,a}(n)=k^i$ if the base-$k$ expansion of $n$ has a suffix of length $i$ consisting entirely of $a$'s and no longer suffix with this property.  As it turns out, we can show that $F$ approximates $V_{k,a}$ up to multiplicative constants.

\begin{lem} There is a natural number $\beta\ge 1$, depending only on $X$ such that for $n\in X$ we have $F(n)\ge V_{k,a}(n)/k^{\beta}$.
\label{lem:below1}
\end{lem}
\begin{proof} 
Since there are only finitely many distinct vectors of sets $(\delta(q,a^r))_{q\in Q}$ for $r\in \N$,
by the pigeonhole principle, there exist natural numbers $r_1,r_2$, with $r_1<r_2$, such that 
$\delta(q,a^{r_1}) = \delta(q,a^{r_2})$ for all $q\in Q$.  We claim that we can take $\beta=r_2$.

If $n=0$ there is nothing to prove, so we may assume $n\in X$ is positive. In addition, if the base-$k$ expansion of $n$ is of the form $a^m$ for some $m\ge 1$, then it follows from the definition that $F(n)\ge V_{k,a}(n)/k=k^{m-1}$.  Thus we may assume that $n$ has at least one base-$k$ digit not equal to $a$. 

Let $wja^m$ denote a base-$k$ of $n$ that is in $L$, where $w\in \Sigma^*$ is a possibly empty word and $j\in \Sigma_k\setminus \{ a\}$.  We claim that $F(n) \ge k^{m-r_2}$.  If $m\le r_2$, the claim is clear, since $F(n)\ge 1$ for all $n\in X$, so we assume that $m>r_2$.

Since $\delta(p,wja^m)=\{p\}$, by definition of $r_1$ and $r_2$, there exists $m_0< r_2$ such that $\delta(p,wja^{m_0})=\{p\}$.  Moreover, since $\delta(p,a)=\{p\}$, we in fact have
$\delta(p,wja^{e})=\{p\}$ for all $e\ge m_0$.

To show that $F(n)\ge V_{k,a}(n)/k^{r_2}= k^{m-r_2}$, let $r\ge 0$ and let $v\in \Sigma_k^*$ be a word of length at most $m-r_2+r$

We must show that the conditions in Equations (\ref{eq:gr}) and (\ref{eq:lr}) hold.

First, if $|v|>r$ then 
$k^r n - [a^{|v|-r}0^r]_k + [v]_k$ has 
base-$k$ expansion of the form $wja^{e}v$ for some $e\ge m_0$ and since $wja^{e}$ is in $L$, we see that 
$wja^e v \in L$ if and only if $v\in L$.  Thus since we are in either case (I) or (II) we see that
$$k^r n - [a^{|v|-r}0^r]_k + [v]_k\in X\iff v\in L$$ in this case.

Next, if $|v|\le r$, then 
$k^r n + [a^{r-|v|} v]_k$ has base-$k$ expansion 
$wja^e v$ with $e\ge m_0$ and so as in the previous case we have that 
$k^r n + [a^{r-|v|} v]_k\in X\iff v\in L$.

Thus $F(n)\ge k^{m-r_2} = V_{k,a}(n)/k^{\beta}$.  
\end{proof}
We now have shown that $F(n)/V_{k,a}(n)$ is bounded below by $1/k^{\beta}$ for some $\beta$ that depends only on $X$ and not on $n$. We will next show that this ratio is bounded above, although this requires our assumption that the set $X$ not be eventually periodic. 

We first require a monotonicity result. 
 \begin{lem} For $w\in L$ and all $i\ge 1$, $F([wa^i]_k) \ge k^i F([w]_k)$.
 \label{lem:multk1}
 \end{lem}
 \begin{proof} Since $w\in L\implies wa\in L$, it suffices to prove the result when $i=1$ and the general result will then follow by induction.
 
 Let $t$ be such that $k^t=F(n)$, where $n=[w]_k$ with $w\in L$.   Then Equations (\ref{eq:gr}) and (\ref{eq:lr}) give that for all $r\ge 0$ and all words $v$ of length 
at most $r+t$ with $|v|=j\le r+t$, we have
$$k^r n - [a^{j-r}\cdot 0^r]_k + [v]_k \in X \iff v\in L$$ if $j>r$; and we have
$$k^r n + [a^{r-j} v]_k \in X \iff v\in L$$ if $j\le r$.

Now let $m=kn+a$, so $m=[wa]_k$.  Then for $r\ge 0$ and a word $v$ of length at most $r+t+1$, if $|v|>r$ then
$$k^r m - [a^{|v|-r} 0^r]+[v]_k = k^{r+1} n + ak^r - [a^{|v|-r} 0^r]_k+[v]_k = k^{r+1}n - [a^{|v| - (r+1)}0^{r+1}]_k +[v]_k,$$
and so we see that this is in $X$ if and only if $v\in L$.  

Similarly, if $|v|\le r$ then 
$$k^r m + [a^{r-j} v]_k = k^{r+1}n + [a^{r+1-j}v]_k,$$ and so this is $X$ if and only if $v\in L.$
It follows that $F(kn+a) = F([wa]_k) \ge kF([w]_k)$, and the general result now follows.     
 \end{proof}

 We now give an argument that shows that when $F(n)$ takes some value, we can find an $n'$ in $X$ such that $F(n')\ge F(n)$ with the property that $n'$ is bounded above by $F(n')$ up to a multiplicative constant.  
 \begin{lem} Let $n\in X$.  Then there is a fixed natural number $\alpha$, depending only on $X$, such that whenever $F(n)= k^y > V_{k,a}(n)$, there exists $n'\in X$ with $F(n')\ge F(n)$, $V_{k,a}(n)=V_{k,a}(n')$, and $n' < k^{\alpha}F(n')$.
 \label{lem:n'1}
 \end{lem}
 \begin{proof}
 We claim we can take $\alpha:=M+2$, where $M$ is the number of states in our minimal automaton $\mathcal{A}$. By definition of $F(n)$ we have $n\ge k^{t-1}$ and so every base-$k$ expansion of $n$ has at least $t$ digits.  We pick a base-$k$ expansion $wu$ of $n$ in $L$ where $w,u\in \Sigma_k^*$ and $|u|=t$.  
When analyzing Equations (\ref{eq:gr}) and (\ref{eq:lr}), the main difficulty that arises is that carries can occur, either when performing additions or subtractions. 

We first note that we may assume the prefix $w$ is non-empty: if not, $n=[u]_k <k^{t+1}$ and so if we take $n'=n$ we have $n'< k^{\alpha} F(n)= k^{\alpha} F(n')$, so the result is clear in this case.

Thus we assume that $w$ is non-empty and in this case it can be written uniquely in the form $w=w_0 c^{\ell}$ for some $\ell\ge 0$, where $c\in \{0,k-1\}$ and $w_0$ is either empty or a word whose last letter is in $\{1,2,\ldots ,k-2\}$.

Thus we now write $n=[w_0 c^{\ell}u]_k$. 
We consider the cases when $w_0$ is empty and when it is not empty separately.
To ease arguments, we write $w_1\sim_L w_2$ to mean that either $w_1$ and $w_2$ are both in $L$ or neither is in $L$. 
\vskip 2mm
{\bf Case A:} $w_0$ is not the empty word.
\vskip 2mm
In this case, we write $w_0 = w_0' s$ where $s\in \{1,2,\ldots ,k-2\}$.  
Since our minimal automaton $\mathcal{A}$ has at most $M$ states, there is some word $w_0''$ of length $<M$ such that $\delta(q_0,w_0')=\delta(q_0,w_0'')$.
We now take $n' = [w_0''s c^{\min(\ell,2)} u]_k$.  Notice that 
$n' < k^{t+M+2}$, so it suffices to show that $F(n')\ge k^t$.

It now suffices to verify that Equations (\ref{eq:gr}) and (\ref{eq:lr}) hold for words $v$ of length $t+r$ for $n'$.  We only consider Equation (\ref{eq:gr}), since Equation (\ref{eq:lr}) is verified in a similar manner.

Since $F(n)=k^t$, by Equation (\ref{eq:gr}) we have
for $r\ge 0$ and $|v|\le t+r$, if $|v|>r$ then
 $$k^r n - [a^{j-r}0^r]_k + [v]_k = k^r[c^{\ell} u]_k - [a^{j-r}0^r]_k + [v]_k$$ is in $X$ if and only if $v\in L$. Moreover, observe that the quantity $k^r [u]_k -  [a^{j-r}0^r]_k + [v]_k$ is greater than or equal to
 $-k^{t+r}+1$ and is at most $2\cdot k^{r+t}-1$.  
 
 We now look at various subcases.
 \vskip 1mm
{\bf Subcase (i):} $k^r [u]_k -  [a^{j-r}0^r]_k + [v]_k<0$, $c=k-1$, and $\ell\ge 2$. 
\vskip 1mm
In this case, 
 $$ k^r[w_0' s c^{\ell} u]_k - [a^{j-r}0^r]_k + [v]_k  
 $$
 either has base-$k$ expansion of the form $w_0' s c^{\ell-1} (c-1) v'$ or of the form $w_0' s c^{\ell} v'$ for some $v'$ which depends on $u$ and $v$ but not on $\ell\ge 2$. In particular, since $k-1$ is idempotent with respect to $\mathcal{A}$ we see that
$w_0' s c^{\ell-1}(c-1) v'\sim_L w_0'' s c (c-1)v'$ and $w_0' s c^{\ell-1}v'\sim_L w_0'' s ccv' $.  Thus
we see that $k^r n'  - [a^{j-r}0^r]_k + [v]_k\in X \iff v\in L$ in this case.
\vskip 1mm
{\bf Subcase (ii):} $k^r [u]_k -  [a^{j-r}0^r]_k + [v]_k<0$, $c=k-1$, and $\ell<2$.
\vskip 1mm
This case is handled similarly to Subcase (i), but where we now use $n'=[w_0'' s c^{\ell} u]_k$.
\vskip 1mm
{\bf Subcase (iii):} $k^r [u]_k -  [a^{j-r}0^r]_k + [v]_k<0$, $c=0$, and $\ell\ge 2$. 
\vskip 1mm
In this case, 
 $$ k^r[w_0' s c^{\ell} u]_k - [a^{j-r}0^r]_k + [v]_k 
 $$
 either has base-$k$ expansion of the form $w_0' (s-1) (k-1)^{\ell} v'$ or of the form $w_0' s c^{\ell} v'$ for some $v'$ which depends on $u$ and $v$ but not on $\ell\ge 2$.  Then using the fact that $c=0$ is idempotent with respect to $\calA$
 we see as in Subcase (i) that if $n'=[w_0'' s 0^{\ell}]_k$ then 
$k^r n'  - [a^{j-r}0^r]_k + [v]_k\in X \iff v\in L$.
\vskip 1mm
{\bf Subcase (iv):} $k^r [u]_k -  [a^{j-r}0^r]_k + [v]_k<0$, $c=0$, and $\ell<2$. 
\vskip 1mm
This case is handled similarly to Subcase (ii), but where we now use $n'=[w_0'' s c^{\ell} u]_k$.
\vskip 1mm
{\bf Subcase (v):}  $k^r [u]_k -  [a^{j-r}0^r]_k + [v]_k\ge 0$, $c=k-1$, and $\ell\ge 2$.
\vskip 1mm
In this case, 
 $$ k^r[w_0' s c^{\ell} u]_k - [a^{j-r}0^r]_k + [v]_k 
 $$
 either has base-$k$ expansion of the form $w_0' (s+1) 0^{\ell} v'$ or of the form
 $w_0' s (k-1)^{\ell} v'$
 for some $v'$ which depends on $u$ and $v$ but not on $\ell\ge 2$.
In this case, we argue as in Subcase (i) with $n'=[w_0'' s c^{2}u]_k$.
\vskip 1mm
{\bf Subcase (vi):}  $k^r [u]_k -  [a^{j-r}0^r]_k + [v]_k\ge 0$, $c=k-1$, and $\ell< 2$.
\vskip  1mm
Here we argue similarly in Subcase (iv) with $n' = [w_0'' s c^{\ell}u]_k$.

\vskip 1mm
{\bf Subcase (vii):}  $k^r [u]_k -  [a^{j-r}0^r]_k + [v]_k\ge 0$, $c=0$, and $\ell\ge 2$.
\vskip 1mm
In this case 
$$ k^r[w_0' s c^{\ell} u]_k - [a^{j-r}0^r]_k + [v]_k  
 $$
 either has base-$k$ expansion of the form $w_0' s 0^{\ell-1} 1 v'$ or of the form
 $w_0' s 0^{\ell} v'$
 for some $v'$ which depends on $u$ and $v$ but not on $\ell\ge 2$.  Then we argue as in the earlier subcases with $n'=[w_0'' s 0^2 u]_k$.
 \vskip 1mm
 {\bf Subcase (viii):}  $k^r [u]_k -  [a^{j-r}0^r]_k + [v]_k\ge 0$, $c=0$, and $\ell<2$.
 \vskip 1mm
 This is done as in the previous subcase with $n'=[w_0'' s 0^{\ell} u]_k$.
 \vskip 2mm
 {\bf Case B:} $w_0$ is the empty word.
 \vskip 2mm
 This case is done similarly to Case A, where we take $n' = [c^{\min(\ell, 2)} u]_k$.
\vskip 2mm
Thus we have found $n'\in X$ with $F(n')\ge k^t$ and $n'\le k^{M+2}k^t.$
 \end{proof}

 \begin{defn}
   We say that two positive integers $n_1, n_2$ are $(k,a)$-{\bf equivalent} if there is some $i\ge 0$ such that either $n_1=k^i n_2 + (k^i-1)a/(k-1)$ or $n_2=  k^i n_1 + (k^i-1)a/(k-1)$.
 \end{defn}
 This notion of $(k,a)$-equivalence is easily seen to be an equivalence relation, and intuitively it says that $n_1$ and $n_2$ are equivalent if the base-$k$ expansion of one can be obtained from the other by either appending a suffix of $a$'s at the end or by deleting a string of $a$'s at the end.

 In the following lemma, we will use Mahler series to prove that an automatic set that is periodic on a sufficiently long interval is necessarily eventually periodic. A power series $h(t)\in \mathbb{Z}[[t]]$ is $k$-{\bf Mahler}, if it satisfies a non-trivial equation
\begin{equation}\label{eq:mahler}
\sum_{i=0}^r P_i(t) h(t^{k^i})=0
\end{equation}
with $P_0P_r\neq 0$ and $P_0,\ldots ,P_r$ polynomials in $t$. If $h(t)$ is the characteristic function of a $k$-automatic set, then $h(t)$ is $k$-regular \cite{AS92} and thus $k$-Mahler \cite{Bec94}.
For a power series $f(t)\in \mathbb{Z}[[t]]$ we write
$f(t)=O(t^N)$ to mean that $f(t)\in t^N\mathbb{Z}[[t]]$.

\begin{lem}
Let $T$ be a $k$-automatic set and let $C\ge 0$.  Then there exists a natural number $N$ depending only on $T$ and $C$, such that if there is some $p \ge 1$ such that
$n+p \in T\iff n\in T$ for $n\in [Cp, Np)$, then $T$ is eventually periodic.
\label{lem:per1}
 \end{lem}
\begin{proof}
Let $h(t)\in \mathbb{Z}[[t]]$ be the characteristic function of $T$.  Then $h$ is $k$-Mahler and so it satisfies a non-trivial $k$-Mahler equation as given in Equation (\ref{eq:mahler}).

The condition that $j+p \in T\iff j\in T$ for $j+p\in [Cp,Np)$ says that $h(t) (1-t^p) = q(t) + O(t^{(N-1)\ell})$ for some polynomial $q$ of degree $<(C+1)p$.  
Then 
$$\sum_{i=0}^r P_i(t) q(t^{k^i})/(1-t^{p k^i}) = O(t^{(N-1)\ell}).$$
Multiplying through by $1-t^{p k^r}$, we then see
$$\sum_{i=0}^r P_i(t) q(t^{k^i})(1-t^{p k^r})/(1-t^{p k^i}) = O(t^{(N-1)p}).$$
Now let $M$ denote the maximum of the degrees of $P_0,\ldots ,P_r$.  Then 
the left-hand side has degree at most $M+k^r p(C+1)$.  Thus if $M+k^{r} p(C+1) < (N-1)p$, then the left-hand side must be identically zero.  In particular, taking 
$N\ge M+k^r(C+2)+1$ gives that the left-hand side vanishes.  But now we claim that 
$h(t)$ must be equal to $q(t)/(1-t^{\ell})$.  To see this, we note that if this is not the case then we may write
$h(t) = q(t)/(1-t^{\ell}) + t^e h_0(t)$ with $h_0(0)\neq 0$ and $$e\ge (N-1)\ell \ge  (M+k^L+1)\ell.$$ Then since $$\sum_{i=0}^r P_i(t) h(t^{k^i})=0$$ and
$$\sum_{i=0}^r P_i(t) q(t^{k^i})/(1-t^{\ell k^i})=0,$$ we must have
$$\sum_{i=0}^r P_i(t) t^{ek^i} h_0(t^{k^i})=0.$$
But notice that all terms in the left-hand side other than $P_0(t) t^{e} h_0(t)$ vanish to order at least $ek$ at $t=0$, while $P_0(t) t^{e} h_0(t)$ vanishes to order at most $t^e + {\rm deg}(P_0)\le e+M$ at $t=0$.  Since $e+M < ke$, we get that 
$h(t)=q(t)/(1-t^p)$, and so $h(t)$ is a rational function whose coefficients lie in $\{0,1\}$ and so the sequence of coefficients of $h(t)$ is eventually periodic.  Thus $T$ is eventually periodic.
\end{proof}
We now show that when $X$ is not eventually periodic that $F(n)/V_{k,a}(n)$ is bounded above. This is the most technical part of the proof of Theorem \ref{thm:reduction1} in cases (I) and (II) and is done via two technical lemmas.
\begin{lem}
\label{lem:vv'}
Suppose that there are natural numbers $M$ and $P$ such that for all words $v, v'$ with $|v'|=|v|=p \le M$ and $[v']_k=P+[v]_k$ we have
$$v \in L \iff v'\in L$$
Then $x\in X\iff x+P\in X$ for all $x\in [kP, k^M-P)$
\end{lem}
\begin{proof}
Suppose that $v$ and $v'$ are words with no leading zeros such that
$[v']_k = [v]_k+P$ and $|v|,|v'| \le \log_k(\min(F(n_i)k^c, F(n_j)))$.
We now consider cases (I) and (II) separately.
\vskip 2mm
{\bf When we are in Case (II):}
In this case, let $v,v'$ be words with no leading zeros such that $[v']_k = [v]_k+P$ and $[v']_k=m$.  If $i=|v|=|v'|$ then by construction 
$$v\in L \iff n_j+P- [a^{|v|}]_k + [v]_k\in X 
\iff n_j - [a^p]_k + [v']_k \in X \iff v'\in L$$ and since we are in case (II), we see that $[v]_k\in X\iff [v']_k\in X$.  On the other hand, if $v$ and $v'$ don't have the same length, then 
we similarly have
$0^{|v'|-|v|}v\in L \iff v'\in L$ and since we are in case (II), $0^{|v'|-|v|}v\not\in L$ and hence $v'\not\in L$ and so $[v']_k\not\in X$ in this case.  Thus we always have
that if $m\in X$ and $m>P$ then $m-P\in X$. Now suppose that $m-P\in X$ but $m\not\in X$.  Now pick the smallest positive integer $s$ such that 
$k^s > P$.  We claim that $X\cap [k^s, k^M)$ is empty. To see this, suppose that this is not the case and pick the smallest $m\in X\cap [k^s, k^M)$.  Since $m\in X\implies m-P\in X$ for $m>P$, we see that $m\in [k^s, k^s+P)$. Hence $m-P<k^s$ and so if $v$ and $v'$ are base-$k$ expansions of $m-P$ and $m$ respectively and $v'$ has no leading zeros, then if $|v|=|v'|$ then $v$ necessarily has a leading zero. And so from the argument above since $v\not\in L$ we then have $v'\not\in L$ and so $m\not\in X$, a contradiction.  Thus 
$X\cap [k^s, k^M)$ is empty and since $k^s\le kP$, $x\in X\iff x+P\in X$ for all $x\in [kP, k^M-P)$.  Thus the result follows when we are in Case II.
\vskip 2mm
{\bf When we are in Case (I):}
\vskip 2mm
Let $m<k^{M}-P$ be a positive integer and let $v,v'$ be words with $[v]_k=m$ and $[v']_k=m+P$.  Since we are in case (I), we may pad $v$ with leading zeros and assume that $|v|=|v'|$.
Then we see that 
$$v\in L \iff n_j+P- [a^{|v|}]_k + [v]_k\in X 
\iff n_j - [a^p]_k + [v']_k \in X \iff v'\in L$$ and since we are in case (I), we see that $[v]_k\in X\iff [v']_k\in X$. Hence we have that $m\in X$ if and only $m+P\in X$ for all $m<k^M-P$. So the result follows in this case.
\end{proof}
 \begin{lem}
     If $X$ is not eventually periodic then $F(n)/V_{k,a}(n)$ is bounded above.
     \label{lem:above1}
 \end{lem}
 \begin{proof}
Suppose, towards a contradiction, that $F(n)/V_{k,a}(n)$ is not bounded above.  Then there exists a sequence $n_1, n_2, n_3, \ldots $ in $X$ with 
$F(n_i)/V_{k,a}(n_i) > k^{i}$ for all $i$. By Lemma \ref{lem:n'1}, we can replace each $n_i$ by a corresponding $n_i'$ and assume additionally that $n_i < k^{\alpha} F(n_i)$ for some fixed $\alpha$ that does not depend on $i$ and depends only on $X$.  

In particular, by Lemma \ref{lem:multk1}, we have
\begin{equation}
    \frac{F\left(k^r n_i + \frac{k^r - 1}{k - 1}a\right)}{V_{k,a}\left(k^r n_i + \frac{k^r - 1}{k - 1}a\right)} \ge k^i
\end{equation}
 for all $r\ge 0$. Moreover, since
 $n_i < k^{\alpha} F(n_i)\in k^{\N}$, we have
 $$k^r n_i + \frac{k^r - 1}{k - 1}a < k^{r+\alpha}F(n_i).$$ Thus by Equation (\ref{eq:i+1}), we see that
 $$F\left(k^r n_i + \frac{k^r - 1}{k - 1}a\right) \le k^{r+\alpha+1}F(n_i)$$
 and so
\begin{equation}
\label{eq:est1}
k^{r} F(n_i)\le F(k^r n_i+(k^r-1)a/(k-1)) \le k^{r+1+\alpha}F(n_i)   
\end{equation} for all $r\ge 0$.  

In particular, if $n_i < n_j$ and $n_i$ and $n_j$ are $(k,a)$-equivalent, then we must have $$n_j = n_i k^r+(k^r-1)a/(k-1)$$ for some $r\ge 0$ and so from the above we have 
\begin{align*}
F(n_j)/V_{k,a}(n_j) &= F(k^r n_i+(k^r-1)a/(k-1))/(k^r \cdot V_{k,a}(n_i))\\
&  \le k^{\alpha+1+r} F(n_i)/(k^r\cdot V_{k,a}(n_i)) \\
& \le k^{\alpha+1} F(n_i)/V_{k,a}(n_i).
\end{align*}
Thus since $n_j\to \infty$ and $F(n_j)/V_{k,a}(n_j)\to\infty$, we 
may refine our sequence $(n_i)$ if necessary that we may assume that they are pairwise $(k,a)$-inequivalent. Now for each positive integer $n$, we let $t(n)$ denote the unique natural number $t$ such that $k^{t-1}\le n < k^{t}$. Then the sequence $(n_i/k^{t(n_i)})$ lies in the compact set $[1/k,1]$ and so by the Bolzano-Weierstrass theorem it has a convergent subsequence.  In particular, since $X$ is not eventually periodic, by Lemma \ref{lem:per1} there is a natural number $N$ such that whenever $p>0$ is such that $x\in X \iff x+p\in X$ for all $x\in [pk, M)$ we necessarily have $M<Np$. 

Choose $\epsilon\in (0,1/2k)$ so that 
\begin{equation}
2k(N+1)\epsilon  <  k^{-\alpha}/2
\end{equation}
we see that we may find $(k,a)$-inequivalent $n_i$ and $n_j$ in this subsequence (by assumption) such that
\begin{equation}
  |  n_i/k^{t(n_i)} - n_j/k^{t(n_j)}| < \epsilon,
  \label{eq:ep1}
\end{equation}
and we may assume in addition that $n_i<n_j$ and that $n_i > (a+1)(1+k\epsilon)/\epsilon$.
We let $$c=t(n_j)-t(n_i).$$
Then, after multiplying both sides of Equation (\ref{eq:ep1}) by $k^{t(n_j)}$, we have 
$$|n_i k^c - n_j| < k^{t(n_j)} \epsilon \le kn_j \epsilon.$$
In particular, 
\begin{equation}\label{eq:k^c}
  k^c < (1+k\epsilon) n_j/n_i
\end{equation}
and
\begin{equation}\label{eq:epsbound}
  |n_i k^c + (k^c-1)a/(k-1) - n_j| < kn_j \epsilon + ak^c < kn_j\epsilon +  a(1+k\epsilon) n_j/n_i.
\end{equation}
Now we let $P:=n_i k^c + (k^c-1)a/(k-1) - n_j$, which is bounded above by $kn_j \epsilon + ak^c$. We consider only the case when $P>0$, with the case when $P<0$ being handled similarly. Then since $F(n_i)\ge k^i$, we see that $F(k^c n_i+(k^c-1)a/(k-1))\ge k^{i+c}$ by Lemma \ref{lem:multk1}.  Thus using Equations (\ref{eq:gr}) and (\ref{eq:lr}) with $r=0$, we see that for all words $v$ with $|v|=p \le \log_k(\min(F(n_i)k^c, F(n_j)))$ we have
$$(k^c n_i + (k^c-1)a/(k-1)) - [a^{p}]_k + [v]_k \in X \iff v \in L$$
Thus by our definition of $P$ we have
$$(n_j + P) - [a^{p}]_k + [v]_k \in X \iff v\in L$$
for $|v| =p \le \log_k(\min(F(n_i)k^c, F(n_j)))$.
But by definition of $n_j$, we also have
$$n_j- [a^{p}]_k + [v]_k \in X \iff v \in L$$ for 
$|v| =p \le \log_k(\min(F(n_i)k^c, F(n_j)))$.
It follows that if $v,v'$ are words of the same length $p\le \log_k(\min(F(n_i)k^c, F(n_j)))$ with $[v']_k=[v]_k+P$  then
$$v\in L \iff n_j+P- [a^{|v|}]_k + [v]_k\in X 
\iff n_j - [a^p]_k + [v']_k \in X \iff v'\in L.$$ 
It now follows from Lemma \ref{lem:vv'} that $x\in X\iff x+P\in X$ for all $x\in [kP, \min(F(n_i)k^c, F(n_j))-P)$. 
Observe that by Equation (\ref{eq:epsbound}) we have
$$P\cdot (N+1) = (k^c n_i + (k^c-1)a/(k-1) - n_j)\cdot (N+1) <
(kn_j \epsilon + a(1+k\epsilon) n_j/n_i) \cdot N.$$
Recall that we refined our sequence of $n_{\ell}$'s earlier and so we have $n_i > (a+1)(1+k\epsilon)/\epsilon$ and thus we have $P(N+1)< 2kn_j (N+1)\epsilon$.

Now by our choice of $n_{\ell}$'s we have
$F(n_j) > k^{-\alpha} n_j$ and by our choice of $\epsilon$ we then have $2k (N+1)\epsilon n_j < F(n_j)$.
We similarly, have 
$$k^c F(n_i) > k^{c-\alpha} n_i > k^{\alpha} (n_j - k n_j\epsilon) < k^{-\alpha} n_j/2,$$ since $\epsilon<1/2k$.
Thus we again have $2k(N+1)\epsilon n_j < k^c F(n_i)$.  Hence $P(N+1) < \min(F(n_j), k^c F(n_i))$.  It follows that 
$x\in X\iff x+P\in X$ for $x\in [kP, PN)]$, but by our choice of $N$ this cannot hold since $X$ is not eventually periodic.
 \end{proof}
We now show that $V_k$ can be defined using the map $F$, continuing with the notation introduced above and under the assumption that $X$ is neither Presburger definable nor sparse.

By Lemmas \ref{lem:below1} and \ref{lem:above1}, the ratio $F(n)/V_{k,a}(n)$ is bounded above and below for $n \in X$, and hence there exists a finite set of integers $\{i_1, \ldots, i_r\}$ such that
\begin{equation}
\label{eq:ratioF}
F(n)/V_{k,a}(n) \in \{k^{i_1}, \ldots, k^{i_r}\}
\end{equation}
for all $n \in X$.

Define $\tilde{X}$ as follows:
\begin{equation}
\label{eq:tildeX}
n \in \tilde{X} \iff n + \frac{k^t - 1}{k - 1} a \in X,
\end{equation}
where $t$ is the largest natural number such that $k^t \le n$.
We note that we can define $X$ from $\tilde{X}$ via the rule
\begin{equation}
\label{eq:Xtotilde}
m\in X\iff m-\frac{k^s- 1}{k - 1} a \in \tilde{X},
\end{equation}
where $s$ is the unique nonnegative integer such that
$$k^s + \frac{k^s- 1}{k - 1} a \le  m < k^{s+1}+\frac{k^s- 1}{k - 1} a.$$
We note that the map from $\tilde{X}$ to $X$ and the map from $X$ to $\tilde{X}$ described in Equations (\ref{eq:tildeX}) and (\ref{eq:Xtotilde}) are inverses of one another and so these maps give bijections between the sets.

We define a map $\tilde{F}: \tilde{X} \to k^{\mathbb{N}}$ by
\begin{equation}
\label{eq:tildeF}
\tilde{F}(n) = F\left(n + \frac{k^t - 1}{k - 1} a\right),
\end{equation}
where $t$ is again the largest natural number such that $k^t \le n$.

Both $\tilde{X}$ and $\tilde{F}$ are definable in the structure $(\mathbb{N}, +, X, k^{\mathbb{N}})$ by construction, and since $X$ is $k$-automatic but not Presburger definable, it follows from Fact \ref{bes_define_kn} that $k^{\mathbb{N}}$ is definable in $(\mathbb{N}, +, X)$. Thus, to complete the proof of Theorem \ref{thm:reduction1}, we assume that $X$ is neither sparse nor eventually periodic and show that $V_k$ is definable using $\tilde{X}$ and $\tilde{F}$.

\begin{lem}
Adopt the notation above and suppose that $X$ is neither sparse not eventually periodic. Then the following hold:
\begin{enumerate}
    \item[(i)] $\tilde{F}(n)/V_k(n)$ is bounded above and below by constants for $n \in \tilde{X}$;
    \item[(ii)] $\tilde{X}$ is not sparse;
    \item[(iii)] if $n \in \tilde{X}$, then $kn \in \tilde{X}$;
    \item[(iv)] $\tilde{F}(kn) \ge k \cdot \tilde{F}(n)$ for $n \in \tilde{X}$.
\end{enumerate}
\label{lem:newS}
\end{lem}
\begin{proof}
Let $n \in \tilde{X}$ and let $t$ be the unique natural number such that $k^t \le n < k^{t+1}$.
We take $m := n + \frac{k^t - 1}{k - 1}a$ and note that $m$ lies in $X$ by definition of $\tilde{X}$.
By definition, 
\[
\tilde{F}(n) = F(m).
\]
Notice if the base-$k$ expansion of $n$ ends with precisely $s$ zeros then the base-$k$ expansion of $m$ ends with exactly $s$ copies of $a$.  Hence
\[
V_{k,a}(m) = V_k(n).
\]
Since $F(m)/V_{k,a}(m)$ is bounded above and below for $m \in X$ by Lemmas~\ref{lem:below1} and~\ref{lem:above1}, we deduce that $\tilde{F}(n)/V_k(n)$ is also bounded above and below. This proves part (i).

To prove part (ii), we note that by Equation (\ref{eq:Xtotilde})
we have a surjective map $h$ from the positive elements of $X$ to the positive elements of $\tilde{X}$ given by
$$m\mapsto m-\frac{k^s- 1}{k - 1} a,$$ where $s$ is the unique integer satisfying $$k^s + (k^s-1)a/(k-1)\le m < k^{s+1}+(k^s-1)a/(k-1).$$ As noted by Equations (\ref{eq:tildeX}) and (\ref{eq:Xtotilde}), the map $h$ is injective and since $h(m)\le m$, we see that the number of elements in $\tilde{X}$ up to a positive integer $x$ is at least as large as the number of elements in $X$ up to $x$, and so $\tilde{X}$ is not sparse since $X$ is not sparse.

For part (iii), suppose $n \in \tilde{X}$, so that $m = n + \frac{k^t - 1}{k - 1}a \in X$. Then
\[
kn + \frac{k^{t+1} - 1}{k - 1}a = k\left(n + \frac{k^t - 1}{k - 1}a\right) + a = km + a.
\]
Since $m \in X$ and $X$ is closed under the map $n \mapsto kn + a$, it follows that $km + a \in X$, and hence $kn \in \tilde{X}$. This proves (iii).

Finally, for part (iv), note from above that
\[
\tilde{F}(n) = F(m), \quad \tilde{F}(kn) = F(km + a).
\]
Since $F(kn + a) \ge kF(n)$ holds for all $n \in X$ and $m \in X$, we conclude that
\[
\tilde{F}(kn) = F(km + a) \ge kF(m) = k\tilde{F}(n).
\]
This completes the proof of part (iv).
\end{proof}
\begin{lem} \label{lem:constructY}
Adopt the notation above.  If $X$ is neither sparse nor eventually periodic, then there exists a subset $Y\subseteq \N$ and a map $H:Y\to k^{\N}$ with the following properties:
\begin{enumerate}
    \item[(i)] $Y$ and $H$ are both definable in $(\N,+,X)$;
    \item[(ii)] for $n\in Y$, $H(n)/V_k(n)$ takes only finitely many distinct values and $H(n)\le V_k(n)$ for all $n\in Y$;
    \item[(iii)] the set $Y_0$ of $n\in Y$ for which $H(n)=V_k(n)$ is not sparse and contains $1$;
    \item[(iv)]  if $n\in Y_0$ then $kn\in Y_0$.
\end{enumerate}
\end{lem}
\begin{proof}
By Lemma \ref{lem:newS} there exists 
a non-sparse 
set $\tilde{X}$ with a definable map $\tilde{F}: \tilde{X}\to k^{\N}$ such that $\tilde{F}(n)/V_k(n)$ is bounded above and satisfies properties (iii) and (iv) in the statement of the lemma. In particular, $\tilde{F}(n)/V_k(n)$ assumes finitely many values for $n\in X$, say $k^{p_1},\ldots ,k^{p_e}$ with $p_1<p_2<\cdots < p_e$. We let $$T_j:=\{n\in \tilde{X}\colon \tilde{F}(n)/V_{k}(n) = k^{p_j}\}.$$  
Then since $\tilde{X}$ is not sparse, at least one set $T_i$ is not sparse and we let $j$ denote the largest index for which $T_j$ is not sparse. Then $T_i$ is sparse for $i>j$ and hence definable in $(\mathbb{N},+,X)$, since $X$ is not Presburger definable.  We now take
$$ T_i': = \{n\in \mathbb{N}\colon \exists m, k^m n \in T_i\}.$$  Then $T_i'$ is again sparse for $i>j$ and therefore definable in $(\mathbb{N},+,X)$. Thus we can define the set $$X_0:=\tilde{X}\setminus \bigcup_{i>j} T_i'$$ in $(\mathbb{N},+,X)$, and for $n\in X_0$ we have $\tilde{F}(n)/V_k(n)$ is bounded above by $k^{p_j}$ and, moreover, this value is achieved on a non-sparse subset of $X_0$.  

We now claim that if $n\in X_0$ and $s\ge 0$ then $k^s n\in X_0$. To see this, suppose that $n\in X_0$ and $k^s n\not\in X_0$. 
Then $n\in X_0\subseteq \tilde{X}$, so $k^s n\in \tilde{X}$ by Lemma \ref{lem:newS}(iii) and thus $k^s n\in \tilde{X}\setminus X_0$. It follows that $k^s n\in T_i'$ for some $i>j$ and so there is some $\ell$ such that $k^{s+\ell} n\in T_i$.  But this gives $n\in T_i'$ by definition and hence $n\not\in X_0$, a contradiction.

Since we can define $k^{\N}$ in $(\N,+,X)$ by Fact \ref{bes_define_kn}, we see we can define the set
$$Y:=X_0\cup k^{\N}$$ and the map $H:Y\to k^{\N}$ given by 
$$H(n)=\tilde{F}(n)\cdot k^{-p_j}$$ for $n\in X_0\setminus k^{\N}$ and $H(n)=n$ for $n\in k^{\N}$. This completes the construction of the set $Y$ and the map $H$, which satisfy properties (i)--(iv). 
\end{proof}

 We can now finally prove our main result of this section.
 Namely, we show that under the following assumptions on $X\subseteq \N$, the structure $(\N,+,k^{\N})$ defines the function $V_k$.
 Those assumptions are that $X$ a $k$-automatic, but not definable in $(\N,+,k^{\N})$, and furthermore there is an automaton $\calA$ recognizing $X$ such that either $\delta(p,0)=\{p\}$ for some state $p$ in $\calA$, or there exist states $p$ and $p'$ in $\calA$ such that $\delta(p,0)=\varnothing$ or $\delta(p,0)=\{p'\}$, where $p'$ is in a different strongly connected component from $p$.
 
\begin{proof}[Proof of Theorem \ref{thm:reduction1} in Cases I and II]
We may assume without loss of generality that $X$ is neither sparse nor eventually periodic. Hence by Lemma \ref{lem:constructY}, we can define a set $Y$ such that properties (i)--(iv) from the statement of the lemma hold.  
So there is a set $Y_0$ (as described in Lemma \ref{lem:constructY}) of $n\in Y$ for which $H(n)=V_k(n)$ contains $1$ and is not sparse. 
Note that this $Y_0$ is not a priori definable in $(\N,+,X)$ since we do not assume that $V_k$ is definable in this structure, though $Y_0$ is clearly $k$-automatic.
By \cite{BHS18}, there is a natural number $P$ such that every natural number is a sum of at most $P$ elements from $Y_0$.

Now we define a map $G:\mathbb{N}\to k^{\N}$ by the rule
\begin{equation}
G(n):= \max_{r\le P} \max_{\{(i_1,\ldots ,i_r)\in Y^r\colon i_1+\cdots +i_r = n\}}  \min(H(i_1),\ldots ,H(i_r)),
\end{equation}
for $n\ge 1$ and where we take $G(0)=1$. By construction, this map $G$ is definable in $(\N,+,X)$, since both $Y$ and the map $H$ are.

We claim that $G(n)=V_k(n)$, which will complete the proof of the theorem, since we will then have defined $V_k$ and all $k$-automatic sets are definable from $V_k$. Notice that if $V_k(n)=k^j$ and if $i_1+\cdots +i_r=n$ with $i_1,\ldots ,i_r\in Y$, then we must have $V_k(i_s) \le k^j$ for some $s$.  Since $H|_Y \le (V_k)|_Y$ we then see that $H(i_s)\le k^j$ and so $\min(H(i_1),\ldots ,H(i_r))\le k^j$.  Thus we see that $G(n)\le V_k(n)$. We now show that $G(n)\ge V_k(n)$. Write $n=k^s n_0$ with $s\ge 0$, $n_0\ge 1$, and $k\nmid n_0$. Then we can write $n_0 = i_1+\cdots + i_r$ for some $r\le P$ and $i_1,\ldots ,i_r\in Y_0$.  
Then 
$i_1 k^s, \ldots, i_r k^s \in Y_0$ and so by definition of the set $Y_0$ we have
$H(i_j k^s) = V_k(i_j k^s)\ge k^s$ for $j=1,\ldots ,r$, and so 
$G(n)\ge  \min(V_k(i_1 k^s), \ldots , V_k(i_r k^s))\ge k^s$. Thus we obtain the desired lower bound.  
This completes the proof for the cases in question.
\end{proof}
\subsection{Proof of Theorem \ref{thm:reduction1} in Case (III)}
In this section we give the proof of Theorem \ref{thm:reduction1} in Case (III).  
To handle Case (III) we require a few basic lemmas.

\begin{lem}
\label{PieceA}
    Let $Z$ be a $k$-automatic subset of $\N$ that is not Presburger definable and let $\mathcal{A} = (Q, \{q_0\}, \Sigma_k, \delta, W)$ be a trim, minimized, deterministic automaton that accepts precisely the words that are base-$k$ expansions of elements of $Z$. Then for each $q\in Q$, the set $Z_q:=[L_{q \to q}]_k$ is definable in $(\N, +, Z)$.
\end{lem}
\begin{proof} Let $E$ denote the language accepted by $\mathcal{A}$. 
Fix $u\in \Sigma_k^*$ such that $\delta(q_0,u)=q$. 
Since $\mathcal{A}$ is minimal, the mapping $q \mapsto \chi_q(w)$, defined by 
\[
\chi_q(w) := 
\begin{cases}
1 & \text{if } \delta(q, w) \in W, \\
0 & \text{otherwise},
\end{cases}
\]
separates the states. Therefore there is a finite set of words $w_1, \dots, w_r \in \Sigma_k^*$ such that each state $q \in Q$ is uniquely determined by the vector $(\chi_q(w_1), \dots, \chi_q(w_r))$.

Our goal, then, is as follows: given a word $v$, consider what happens when $\calA$ is run on $uvw_i$ for each $i$. If $v$ is in $L_{q \to q}$, then the run of $uvw_i$ will begin at $q_0$, proceed to $q$, perform some cycle on the word $v$ that ends back at $q$, and then end at $\delta(q, w_i)$. If $v$ is \textit{not} in $L_{q \to q}$, then $uvw_i$ will instead end at $\delta(q', w_i)$, where $q' = \delta(q, v)$. By our choice of $w_i$, this means that whether $v \in L_{q \to q}$ can be determined by whether each $uvw_i$ is accepted by $\calA$, or equivalently, whether each $[uvw_i]_k$ is in $Z$.

It suffices, therefore, to show that the relation between $[v]_k$ and the tuple $([uvw_i]_k)_i$ is definable in $(\N, +, Z)$, as then, we can define whether $[v]_k \in Z_q$ by deciding whether for each $i$ we have $[uvw_i]_k \in Z$ if and only if $\delta(q, w_i) \in W$. Note that this relation is not a function, because $[\bullet]_k$ is not one-to-one, due to the possibility of leading zeros. We will therefore instead define, for each $i$, the ternary relation of triples $([v]_k, [10^{|v|}]_k, [uvw_i]_k)$, from which the previous relation is definable. This is given by the following formula in $(x, y, z)$:

$$y \in k^\N \wedge y > x \wedge z = ([u]_k \cdot y) + (k^{|w_i|} \cdot x) + [w_i]_k.$$

Note that the multiplications by $[u]_k$ and $k^{|w_i|}$ are constant multiplications and hence Presburger-definable. The only remaining matter is how we mean to use $y \in k^\N$ in the above formula. But because $Z$ is $k$-automatic and not Presburger-definable, $k^\N$ is definable in $(\N,+,Z)$ by Fact \ref{bes_define_kn}.
\end{proof}

\begin{lem}
\label{lem:lastone}
Let $\mathcal{A} = (Q, \{q_0\}, \Sigma_k, \delta, W)$ be a trim, minimized, deterministic automaton and 
Let $q,q',q''\in Q$ be in same strongly connected component.  Then we have:
\begin{enumerate}
\item[(a)] If $[L_{q'\to q'}]_k$ is definable in $(\N,+,k^{\N})$ then so is $[L_{q\to q''}]_k$;
\item[(b)] if $[L_{q'\to q'}]_k$ is sparse then so is $[L_{q\to q''}]_k$.
\end{enumerate}
\end{lem}
\begin{proof}
We will once again use Fact \ref{bes_define_kn} to work in the expanded language $\{+,<,k^\N, \ldots\}$.
To prove (a), since $q'$ and $q$ are in the same strongly connected component, there is some word $u\in L_{q'\to q}$ and so if $w$ is a word with no leading zeros, then $[w]_k\in [L_{q\to q'}]_k$ if and only if $0^i w\in L_{q\to q'}$ for some $i$. By the pumping lemma, there is some pumping length $p$ such that we can take $i<p$.  Then
$$[w]_k\in [L_{q\to q'}]_k \iff \exists i<p,~~ 0^i w\in L_{q\to q'} \iff \exists i<p,~~ u 0^i w \in L_{q'\to q'}$$
Then $n=[w]_k\in [L_{q\to q'}]_k$ if and only if 
$$\exists s\in k^\N, (n< s \le (k^{p} \cdot n)) \wedge (([u]_k \cdot s) + n\in [L_{q'\to q'} ]_k)$$ 

(here $k^{p}$ and $[u]_k$ are constants).
Since $[L_{q'\to q'}]_k$ is definable in $(\N,+,k^{\N})$,
we then see that $[L_{q\to q'}]_k$ is definable in $(\N,+,k^\N)$.

Similarly, let $w'$ be a word such that $\delta(q'',w')=q'$, which is guaranteed to exist because $q', q''$ are in the same strongly connected component. Let $v$ be a word without leading zeros that may or may not be in $L_{q\to q''}$, and let $n = [v]_k$. Then:

$$n\in [L_{q\to q''}]_k \iff vw' \in L_{q \to q'} \iff k^{|w'|} n + [w']_k\in [L_{q\to q'}]_k.$$
So $L_{q\to q''}$ is also definable in $(\N,+,k^{\N})$.

To prove (b), observe that by by \cite[Prop. 7.1(3)]{BM19}, if $L_{q \to q'}$ is not sparse then there exist words $u,v,a,b$ with $a,b$ distinct and of the same length such that $L_{q\to q'}\supseteq u\{a,b\}^*v$.  But now since $q,q'$ are in the same strongly connected component, there exists $w\in L_{q'\to q}$ and $w' \in L_{q'\to q''}$ and so 
 $L_{q'\to q'}\supseteq wu\{a,b\}^*vw'$, which, again using  \cite[Prop. 7.1(3)]{BM19}, contradicts that $L_{q'\to q'}$ is sparse.   
\end{proof}

We are finally able to prove Theorem \ref{thm:reduction1} in Case (III). As a reminder, this is the case where we assume that $0$ is idempotent with respect to $\calA$ and that $\delta(p, 0)$ contains a distinct state $p'$ in the same strongly connected component as $p$.

\begin{proof}[Proof of Theorem \ref{thm:reduction1} in Case (III)]
We let $p'\in Q$ be such that $\delta(p,0)=\{p'\}$.  Since $0$ is idempotent we have $\delta(p',0)=\{p'\}$ and by assumption $p$ and $p'$ are in the same strongly connected component.
Since $\delta(p',0)=\{p'\}$, we see that the cycle language $L_{p'\to p'}$ falls into Case I, and hence is either definable in $(\N,+,k^{\N})$, or it defines $V_k$.  
If expanding $(\N,+)$ by a predicate for $[L_{p'\to p'}]_k$ generates the same definable sets as $(\N,+,k^{\N})$, then by taking the subautomaton induced by this strongly connected component containing $p'$ (but where we take our initial state to be $p$ and our final state to be $p$) we see by Lemma \ref{PieceA} that $[L_{p\to p}]_k$ defines $[L_{p'\to p'}]_k$ over $(\N,+)$, and hence also defines $V_k$.  
Thus we may assume that $[L_{p'\to p'}]$ is definable in $(\N,+,k^{\N})$, and hence by Lemma \ref{lem:lastone} so is $[L_{p\to p}]_k$. 
Thus we have proven the dichotomy in this case.
 \end{proof}

\section{Proof of dichotomy: general case}\label{sec:gen}
We now show how to prove Theorem \ref{thm:main} from Theorem \ref{thm:reduction1}.  More precisely we prove that if $X$ is a $k$-automatic set that is not definable in $(\mathbb{N},+,k^{\mathbb{N}})$, then $X$ defines $V_k$. 

We need two lemmas before giving the proof of the general dichotomy.

\begin{lem}
\label{missing_transition_forbidden_subword}
	Let $\calA = (Q, \{q_0\}, \Sigma_k, \delta, F)$ be a deterministic, trim, strongly connected automaton. If there exists a state $q \in Q$ and a character $c \in \Sigma_k$ such that $\delta(q, c) = \varnothing$, then there exists a ``forbidden subword'' for $\calA$, i.e. a word $w$ such that no word accepted by $\calA$ contains $w$ as a subword.
\end{lem}
\begin{proof}
	We will construct $w$ iteratively. Begin by letting $w_0 = \varepsilon$. Enumerate the states in $Q$ as $Q = \{q_1, \dots, q_n\}$.
	
	Our goal is to construct each $w_i$ such that $\delta(q_j, w_i) = \varnothing$ for each $1 \leq j \leq i$. So to construct $w_i$ from $w_{i-1}$, consider $\delta(q_i, w_{i-1})$, i.e. determine which state $p_i$, if any, results from running the word $w_{i-1}$ starting from $q_i$. (By determinism, there is at most one such state.) If there is no such state, we may let $w_i = w_{i-1}$. Otherwise, note that by strong connectedness and determinism of $\calA$, there is a word $v_i$ such that $\delta(p_i, v_i) = \{q\}$.
	
	We then let $w_i = w_{i-1} v_i c$. Note that:
	
	\begin{align*}
		\delta(q_i, w_i) &= \delta(\delta(\delta(q_i, w_{i-1}), v_i), c) \\
		&= \delta(\delta(p_i, v_i), c) \\
		&= \delta(q, c) \\
		&= \varnothing.
	\end{align*}
	
	Moreover, note that suffixing words onto $w_{i-1}$ does not change the fact that $\delta(q_j, w_i) = \delta(q_j, w_{i-1}) = \varnothing$ for $j < i$. In other words, we achieve the desired effect that $\delta(q_j, w_i) = \varnothing$ for each $1 \leq j \leq i$.
	
	We claim that no word accepted by $\calA$ contains $w_n$ as a subword. Consider any word of the form $uw_nv$, and let $i$ be such that $\delta(q_0, u) = q_i$. Then when $\calA$ is run on this word:
	
	\begin{align*}
		\delta(q_0, uw_nv) &= \delta(\delta(\delta(q_0, u), w_n), v) \\
		&= \delta(\delta(q_i, w_n), v) \\
		&= \delta(\varnothing, v) \\
		&= \varnothing.
	\end{align*}
	
	So the word $uw_nv$ is rejected by $\calA$.
\end{proof}

For the next lemma, recall that in the context of Semenov's characterization defined in \S \ref{ss:sem}, $\Sigma_{\ell,m,c}$ is the set of all words of length a multiple of $\ell$ that are the base-$k$ expansion (\textit{without} leading zeros) of a word congruent to $-c$ modulo $m$.

\begin{lem}
\label{periodic_no_forbidden_subword}
	Let $\ell, m, c$ be integers with $m, \ell > 0$, and let $u \in \Sigma_k^*$. Then the language $L = u \Sigma_{\ell, m, c}$ has no ``forbidden subword,'' i.e. every $w \in \Sigma_k^*$ is a subword of some word in $L$.
\end{lem}
\begin{proof}
	Given $w \in \Sigma_k^*$, we will find an element of $L$ containing $w$ as a subword. Pick $M$ such that $\ell M > |w|$ and such that $(k-1) k^{\ell M - |w| - 1} \geq m+1$. Now consider the language $M = uw\Sigma_k^{\ell M - |w|}$. Every word in $M$ has $w$ as a subword, so it will suffice to show that $L \cap M$ is nonempty.
	
	By definition of $L$ and $M$, $L \cap M$ will contain precisely those words $uwv$ such that $|v| = \ell M - |w|$ and $wv \in \Sigma_{\ell, m, c}$. Note that the length condition of $\Sigma_{\ell, m, c}$ is satisfied, because $|wv| = |w| + |v| = |w| + \ell M - |w| = \ell M$ is a multiple of $\ell$. The remaining condition on $v$ is that $wv$ must be the base-$k$ expansion without leading zeros of a word congruent to $-c$ modulo $m$. In other words, $L \cap M$ will be nonempty as long as there is a word $v$ of length $\ell M - |w|$ beginning with a nonzero character and such that $[wv]_k \equiv -c \pmod{m}$.
	
	Consider the set:
	
	$$S = \{[wv]_k : v \in \Sigma_k^{\ell M - |w|}, v \text{ does not begin with } 0\},$$
	
	and observe that we may write:
	
	$$S = \Z \cap [[w10^{\ell M - |w| - 1}]_k, [w(k-1)^{\ell M - |w|}]_k),$$
	
	i.e. $S$ is a set of consecutive integers. By a counting argument, $|S| = (k-1) k^{\ell M - |w| - 1} \geq m+1$. Because $S$ is a set of more than $m$ consecutive integers, it follows that some $s \in S$ is congruent to $-c$ modulo $m$. So by letting $v$ be such that $[wv]_k = s$, we obtain a word $uwv \in L \cap M$.
\end{proof}

We now prove Theorem \ref{thm:main}.

\begin{proof}[Proof of Theorem \ref{thm:main}] Assume towards a contradiction that there is some $k$-automatic set $X$ that is not definable in $(\N,+,k^\N)$ but which does not define $V_k$, and pick a minimal DFA $\mathcal{A}$ that accepts the language $L$ consisting of words whose base-$k$ expansion lies in $X$.  By Lemma \ref{PieceA}, we have that $X$ defines $X_q:=[L_{q\to q}]_k$ for all $q\in Q$.  Hence by Theorem \ref{thm:reduction1}, since $X$ does not define $V_k$, we have that $X_q$ is definable in $(\N,+,k^{\N})$ for all $q\in Q$. 

We let $C_1,\ldots, C_s$ denote the strongly connected components of $\mathcal{A}$. Recall that some $C_i$ are non-leaves, which by Definition \ref{def:digraph} are strongly connected components containing a transition to a different strongly connected component. We claim that if $q \in C_i$ for such a component, $X_q$ must be sparse.

To this end, fix such an $i$ and $q \in C_i$, and let $p \in C_i, y \in \Sigma_k$ be such that $\delta(p, y)$ lies outside $C_i$. Then $X_q$ is recognized by the automaton formed by taking $C_i$ as the set of states, modifying $\delta$ to a new transition function $\delta'$ containing only transitions within $C_i$, and making $q$ the only initial and final state; in particular, note that $\delta'(p, y) = \varnothing$.

By Lemma \ref{missing_transition_forbidden_subword}, there is a word $w$ that is not a subword of any word accepted by this new automaton, and therefore by the contrapositive of Lemma \ref{periodic_no_forbidden_subword}, the language of $X_q$ contains no subset of the form $u\Sigma_{\ell,m,c}$. We saw earlier that $X_q$ was definable in $(\N, +, k^\N)$, so by Fact \ref{semenov}, $X_q$ is a union of sparse sets and is thus sparse. Furthermore, by Lemma \ref{lem:lastone}, $L_{q\to q'}$ must be sparse for all $q,q'\in C_i$.

Now, to each word $w$ in $L$, we associate a finite sequence of states $(p_1, q_1, p_2, q_2, \dots, p_n, q_n)$ and a finite sequence of characters $(\sigma_1, \dots, \sigma_{n-1})$ such that the run of $w$ in $\calA$:

\begin{itemize}
    \item begins in the initial state $p_1$;
    \item proceeds to $q_1$, in the same strongly connected component as $p_1$;
    \item transitions on the character $\sigma_1$ to $p_2$, in a different strongly connected component;
    \item proceeds to $q_2$, in the same strongly connected component as $p_2$;\\ \vdots
    \item finishes in $q_n$, an accept state.
\end{itemize}

There are only finitely many choices of the finite sequences $(p_1, q_1, p_2, q_2, \dots, p_n, q_n)$ and $(\sigma_1, \dots, \sigma_{n-1})$, so we see that $L$ is a finite union of languages of the form
    $$M = L_{p_1 \to q_1} \sigma_1 L_{p_2 \to q_2} \sigma_2 \dots L_{p_{n-1} \to q_{n-1}} \sigma_{n-1} L_{p_n \to q_n},$$
    where:
    \begin{itemize}
    \item $p_i$ and $q_i$ are in the same strongly connected component;
    \item $\sigma_i\in \Sigma_k$ and $\delta(q_i,\sigma_i)=p_{i+1}$ for $i=1,\ldots ,n-1$; and
    \item $q_i, p_{i+1}$ are in different strongly connected components.
\end{itemize}

We therefore aim to achieve a contradiction by proving that for each such language $M$, the corresponding set $[M]_k$ is definable in $(\N, +, k^\N)$. This will show that $X$ itself is definable in $(\N, +, k^\N)$, directly contradicting a premise.

Consider one such $M = L_{p_1 \to q_1} \sigma_1 L_{p_2 \to q_2} \sigma_2 \dots L_{p_{n-1} \to q_{n-1}} \sigma_{n-1} L_{p_n \to q_n}$, and note that each $p_i$ is in the same strongly connected component as the corresponding $q_i$. Recall that each $X_{p_i}$ is definable in $(\N, +, k^\N)$; by Lemma \ref{lem:lastone}, so is $L_{p_i \to q_i}$. Moreover, note that $L_{p_1 \to q_1}, \dots, L_{p_{n-1} \to q_{n-1}}$ are path languages within strongly connected components of $\calA$ that are not leaves, in which case we showed earlier that each of them is sparse.

Therefore, the language $M' = L_{p_1 \to q_1} \sigma_1 L_{p_2 \to q_2} \sigma_2 \dots L_{p_{n-1} \to q_{n-1}} \sigma_{n-1}$ is sparse, because it is a concatenation of sparse languages (cf. \cite[Prop. 7.1(6)]{BM19}). So $M = M'L_{p_n \to q_n}$ is the concatenation of a sparse language and a language whose corresponding subset $[L_{p_n \to q_n}]_k \subseteq \N$ is definable in $(\N, +, k^\N)$. By Fact \ref{semenov}, $[M]_k = [M'L_{p_n \to q_n}]_k$ is also definable in $(\N, +, k^\N)$, and so as stated above, we achieve our contradiction, because $X$ is shown to be a finite union of sets definable in $(\N, +, k^\N)$.
\end{proof}

\section{Further Results}\label{sec:fur}

The characterization established in the previous section has consequences in the realms of both definability and decidability.
B\`es previously established in \cite{Bes97} that for $k$ and $\ell$ multiplicatively independent natural numbers, the structure $(\N,+,V_k,\ell^{\N})$ defines multiplication and hence is undecidable.
Using a theorem of the third-named author showing that under the same hypotheses $(\N,+,k^{\N},\ell^{\N})$ does not define multiplication \cite{Sch25}, we extend that result with the following corollary.

\begin{cor}
    If $k, \ell \in  \N_{>1}$ are multiplicatively independent, and $X \subseteq \N$ is $k$-automatic and $Y \subseteq \N$ is $\ell$-automatic, then the structure $(\N,+,X,Y)$ defines multiplication if and only if either $X$ is not definable in $(\N,+,k^{\N})$ or $Y$ is not definable in $(\N,+,\ell^{\N})$.
\end{cor}

Using results of Hawthorne \cite{Haw20} relating definability in $(\Z,+,k^{\N})$ to stability of the structure $(\Z,+,X)$ where $X \subseteq \N$ is $k$-automatic, we also obtain a further characterization of definability of $V_k$ in such expansions of $(\Z,+)$.

\begin{cor}
    Suppose that $X \subseteq \N$ is $k$-automatic. Then exactly one of the following holds:
    \begin{enumerate}
        \item The structure $(\Z,+,X)$ is stable;
        \item The structures $(\Z,+,X)$ and $(\Z,<,+,k^{\N})$ define the same sets;
        \item The structure $(\Z,+,X)$ defines $V_k$.
    \end{enumerate}

    \begin{proof}
    First, we note that in \cite{Haw20}, Hawthorne shows that if $X \subseteq \N$ is $k$-automatic the structure $(\Z,+,X)$ is stable if and only if $X$ is sparse.
    In the case that $X$ is not sparse, by results of Bell, Hare, and Shallit in \cite{BHS18}, there exist natural numbers $m,d,N$ such that every natural number $n>N$ that is a multiple of $d$ can be written as the sum of at most $m$ elements of $X$.
    From this we can define $\N$ in $(\Z,+,X)$. 
    Having defined $\N$, we may now apply Theorem \ref{thm:main} to the induced structure on $\N$.
    Hence we conclude that either $(\Z,+,X)$ defines the same sets as $(\Z,<,+,k^{\N})$ in the case that $\N \setminus X$ is sparse, or else $(\Z,+,X)$ defines $V_k$.
    \end{proof}
\end{cor}

This corollary prompts us to ask the following question.

\begin{ques}
    If $X \subseteq \Z$ is recognized by an automaton with alphabet $\Sigma = \{-,0,\ldots ,k-1\}$, does the above trichotomy hold?
\end{ques}

We note that for $X \subseteq \Z^2$, this is not the case.
For $k\in \N_{>1}$, consider $<_k:=\{(x,y)\in \Z^2: V_k(|x|)<V_k(|y|)\}$.
By results of \cite{AD19}, the structure $(\Z,+,<_k)$ is dp-minimal for all $k>1$,
so $V_k$ cannot be defined in this structure.
However, the classic ordering $<$ of the integers is also not definable in this structure, yet the structure is nevertheless unstable.

It is also natural to ask whether Theorem \ref{thm:main} holds when we consider a $k$-automatic set $X\subseteq \N^m$ with $m>1$.
One of the obstacles in proving the higher-arity analog is the fact that Semenov only gives a nice characterization of one dimensional definable sets in \cite{Sem80}.
This prompts the following question.
\begin{ques}
Suppose $X\subseteq \N^m$, with $m>1$, is $k$-automatic.
If the unary sets definable in $(\N,+,X)$ are precisely the unary sets in $(\N,+,k^{\N})$, is $X$ necessarily definable in $(\N,+,k^{\N})$?
\end{ques}

As illustrated by the structure $(\Z,+,<_k)$, this question has a negative answer when we replace $X \subseteq \N^m$ with $X \subseteq \Z^m$ (for $m>1$).

It is also quite natural to ask whether we can effectively decide whether a given $k$-automatic set $X \subseteq \N$ is definable in $(\N,+,k^{\N})$.
As demonstrated by our results, this boils down to deciding whether there exists some $N \in \N$
 after which $\{x \in X:x>N\}$ is periodic.
 Certainly we can write down an algorithm that enumerates all $X$ definable in $(\N,+,k^{\N})$, but deciding whether $X$ is definable in $(\N,+,k^{\N})$ for an arbitrary $X$ definable in $(\N,+,V_k)$ seems strictly trickier.
\bibliographystyle{plain}
\bibliography{biblio}

\end{document}